\documentclass[preprint,authoryear,12pt]{elsarticle}
\usepackage{amsfonts,amssymb,amsmath,amsopn}
\allowdisplaybreaks[4]
\usepackage{graphics}
\usepackage{graphicx}
\usepackage{rotating}
\usepackage{bigints}
\usepackage{booktabs}
\usepackage{subfigure}
\usepackage{url}
\usepackage{lscape}
\usepackage{amssymb}
\usepackage{amsthm}
\usepackage{color}
\usepackage{enumerate}
\usepackage{float}
\usepackage[colorlinks=true]{hyperref}
\usepackage{setspace}
\usepackage{multirow}
\usepackage[table]{xcolor}
\usepackage{natbib} \bibpunct{(}{)}{;}{a}{,}{,} 
\usepackage{enumitem}
\newcommand{\ignore}[1]{}
\newtheorem{asu}{{\sc Assumption}}
\newtheorem{corollary}{{\sc Corollary}}
\newtheorem{thm}{\sc Theorem}
\newtheorem{rem}{\sc Remark}
\newtheorem{pro}{\sc Proposition}
\newtheorem{lem}{\sc Lemma}

\newtheorem{ex}{\sc Example}
\usepackage{xcolor}

\usepackage[margin=1in]{geometry}
\newproof{pf}{Proof}
\newproof{pot1}{Proof of Theorem \ref{thm1}:}
\newproof{pot2}{Proof of Theorem \ref{thm2}}
\newproof{pot3}{Proof of Theorem \ref{thm3}}

\journal{}

\makeatletter
\renewenvironment{pot1}[1][\sc Proof of Theorem \ref{thm1}]{
	\par\pushQED{\qed}%
	\normalfont \topsep6\p@\@plus6\p@\relax
	\trivlist
	\item[\hskip\labelsep
	#1\@addpunct{:}]\ignorespaces
}{%
	\popQED\endtrivlist\@endpefalse
}
\makeatother

\makeatletter
\renewenvironment{pot2}[1][\sc Proof of Theorem \ref{thm2}]{
	\par\pushQED{\qed}%
	\normalfont \topsep6\p@\@plus6\p@\relax
	\trivlist
	\item[\hskip\labelsep
	#1\@addpunct{:}]\ignorespaces
}{%
	\popQED\endtrivlist\@endpefalse
}
\makeatother

\makeatletter
\renewenvironment{pot3}[1][\sc Proof of Theorem \ref{thm3}]{
	\par\pushQED{\qed}%
	\normalfont \topsep6\p@\@plus6\p@\relax
	\trivlist
	\item[\hskip\labelsep
	#1\@addpunct{:}]\ignorespaces
}{%
	\popQED\endtrivlist\@endpefalse
}

\makeatother

\begin{document}
	\begin{frontmatter}
		\title{On the Realized Joint Laplace Transform of Volatilities with Application to Test the Volatility Dependence}
			\author{\sc{Xinwei Feng}}
			\address{Zhongtai Securities Institute for Financial Studies, Shandong University}
			\author{\sc{Yu Jiang}\corref{cor1}}
			\address{Department of Mathematics, University of Macau}
			\author{\sc{Zhi Liu}}
			\address{Department of Mathematics, University of Macau}
			\author{\sc{Zhe Meng}}
			\address{Zhongtai Securities Institute for Financial Studies, Shandong University}
			\cortext[cor1]{Corresponding author.Email: yc27959@connect.um.edu.mo. Xinwei FENG acknowledges the financial support from the National Natural Science Foundation of China (12371148, 12001317), the Shandong Provincial Natural Science Foundation (ZR2020QA019), and the QILU Young Scholars Program of Shandong University.}
\begin{abstract}
In this paper, we first investigate the estimation of the empirical joint Laplace transform of volatilities of two semi-martingales within a fixed time interval $[0, T]$ by using overlapped increments of high-frequency data. The proposed estimator is robust to the presence of finite variation jumps in price processes. The related functional central limit theorem for the proposed estimator has been established. Compared with the estimator with non-overlapped increments, the estimator with overlapped increments improves the asymptotic estimation efficiency. Moreover, we study the asymptotic theory of estimator under a long-span setting and employ it to create a feasible test for the dependence between volatilities. Finally, simulation and empirical studies demonstrate the performance of proposed estimators.
		\end{abstract}
	\end{frontmatter}
	
	\section{Introduction}
In recent years, blooming global commerce and modern digital technologies have produced substantial volumes of data. Nevertheless, analyzing large-scale, high-dimensional datasets remains challenging. With the wide availability of high-frequency data, research on financial econometrics has been rapidly developed. Among others, volatility estimation is one of the most popular topics due to its widespread use in option pricing, risk management, and high-frequency trading, see \cite{barndorff2002econometric}, \cite{jacod2008asymptotic} and \cite{andersen2010continuous}. Statistical inference for the volatility is inherently complex since most of the studies assume that the asset prices follow It$\hat{\text {o}}$ semi-martingale and then the underlying volatility process is latent, see \cite{merton1973theory}, \cite{heston1993closed}, \cite{bates1996jumps}, \cite{andersen2001distribution}, \cite{tijms2003first} and the references therein. Fortunately, The recent accessibility of high-frequency data has made it feasible to conduct an analysis of the volatility of the price process within a relatively short period. Under the assumption of continuous time models, there are usually two ways to define volatility. The first is called spot volatility, which refers to the value of the coefficient that describes the diffusion process of price return fluctuating at a specific point in time. If the daily fluctuation of asset prices is considered, Spot volatility does not play a key role. This leads to another type of volatility, Integrated volatility. Integrated volatility is the integration of Spot volatility within a period (such as a day). Spot volatility and Integrated volatility have been widely studied, see, e.g., \cite{andersen2001distribution}, \cite{barndorff2002econometric}, \cite{fan2008spot}, \cite{jacod2008asymptotic}, \cite{jacod2009microstructure},
\cite{Mancini2009non}.

Therefore, how to use high-frequency data to obtain more information about volatility has become a concern for statisticians and financial economists. \cite{todorov2012realized} first proposed the realized Laplace transform (RLT) of volatility, which is a consistent estimator of the empirical Laplace transform (ELT). The empirical Laplace transform of volatility contains more information than the integrated volatility. For instance, the empirical Laplace transform of volatility is a mapping from the data to a random function, while the integrated volatility is only a mapping from the data to a random variable. Hence, the integrated volatility can be obtained from the empirical Laplace transform of volatility. Moreover, the empirical Laplace transform of volatility preserves information about the characteristics of volatility under some mild stationarity conditions. This article has enabled scholars to study volatility beyond the statistical estimation level and has opened up the exploration of the distribution properties of volatility.
Shortly thereafter, \cite{li2013volatility} introduced a novel measure known as volatility occupation time to investigate the distribution of volatility. There are still many researchers who have made many outstanding contributions in these two directions. For example, \cite{todorov2011realized},\cite{todorov2012realizedb} further improved the theory of realized Laplace transform of volatility under the pure-jump model setting. \cite{wang2019rate}, \cite{wang2019realized} developed the limiting behavior of realized Laplace transform of volatility with the presence of microstructure noise. The bootstrap inference of the realized Laplace transform of volatility has been studied by \cite{hounyo2023bootstrapping}, and the large deviation principle of the realized Laplace transform of volatility has been derived by \cite{feng2022large}. On the other hand, \cite{li2016estimating}
estimated the volatility occupation time employing the technology of Laplace inversion transform. \cite{christensen2019realized} exploit the results of realized empirical distribution of volatility and construct new goodness-of-fit tests. Related research on Occupation density can be found in \cite{zhang2022occupation} and so forth. These researches are sufficient to show that the study of the distribution of volatility and its properties are of great significance to the field of statistics and finance.

The aforementioned works, however, only consider the univariate case. In the context of financial econometrics, the dependent structure of price and volatility among assets are also important in risk management, portfolio allocation, etc. For the asset price, \cite{andersen2001distribution} first proposed the estimator of covariance of multivariate price processes. With the prosperous development of financial industries in recent years, many financial innovations involve complex derivations and some structured financial products, such as Collateralized Debt Obligation (CDOs). The past several years have witnessed an increasing number of research focusing on the correlation structure of two or more assets, for example, \cite{andersen2003modeling}, \cite{hayashi2005covariance}, \cite{ait2010high}, \cite{dalalyan2011second}, \cite{hayashi2011nonsynchronous}, and \cite{kong2020asymptotics}, among others. \cite{ding2025multiplicative} propose a multiplicative volatility factor model (MVF) to capture the co-movements of volatilities simultaneously, which greatly enhances the volatility forecasting accuracy in empirical study. Consequently, it is very helpful for understanding the joint distribution of multivariate volatilities and the relationships among the volatilities of different assets to manage and control financial risk. For instance, we aim to determine whether the distribution of volatility of two assets in different sectors or the same sector across different time periods exhibits similarity or equivalence. Additionally, we seek to ascertain whether the volatility of different assets is independent and so on. These problems need to be urgently solved both in the theory of statistical inference and in the practice of quantitative finance. However, they are rarely involved in current research. Therefore, our work will be a strong supplement and in-depth exploration of these issues.

In financial modeling, the volatility is assumed to be stochastic. Therefore, the dependence structure among the volatilities of different assets is also interesting. To our knowledge, no work in literature considers the dependence of volatilities. We attempt to employ the empirical joint Laplace transform of two volatilities to study their dependence. Precisely,
we define the empirical joint  Laplace transform of volatilities as
\begin{eqnarray}
\int_0^T\mbox{e}^{- \langle(u, v), ((\sigma_{s}^{X})^2, (\sigma_{s}^{Y})^2)\rangle}\mathrm{d}s,
\end{eqnarray}
where, $(\sigma^{X})^2$ and $(\sigma^{Y})^2$ are the two volatility processes of two assets $X$ and $Y$, respectively. Recall that the joint Laplace transform contains full information on the distribution of a bivariate random vector. Hence, it can be used to detect the dependence of two random variables. For instance, if we consider a long enough time interval $[0, T]$, then under some appropriate stationarity and mixing conditions, as stated in \cite{DD2005}, the empirical joint  Laplace transform $\frac{1}{T}\int_0^T\mbox{e}^{- \langle(u, v), ((\sigma_{s}^{X})^2, (\sigma_{s}^{Y})^2)\rangle}\mathrm{d}s$ will be close to the joint Laplace transform of $(\sigma_{t}^X)^2$ and $(\sigma_{t}^Y)^2$, namely, $E[e^{-u(\sigma_{t}^X)^2-v(\sigma_{t}^Y)^2}]$. It thus can be used to either recover the joint distribution of two volatilities, study their independence, or investigate any other quantities related to the joint distribution.

Since the empirical joint  Laplace transform $\int_0^T\mbox{e}^{- \langle(u, v), ((\sigma_{s}^{X})^2, (\sigma_{s}^{Y})^2)\rangle}\mathrm{d}s$ is unobservable, hence in this paper we will firstly construct a consistent estimator for fixed time interval $[0, T]$ by using high-frequency data. Secondly, we will derive the asymptotic behavior under a long time span.

The rest of the paper is organized as follows. In Section \ref{setup}, we introduce the model assumptions and present the estimators. Section \ref{asymptotic} derives the consistency and central limit theorem of the proposed estimators for fixed time interval $[0, T]$. The asymptotic behavior of the estimator under a long time period is presented in Section \ref{longspan}. In Section \ref{simulation}, \ref{empirical}, we exhibit the simulation results and empirical study, respectively. We conclude the paper in Section \ref{conclusion}, and all technical proofs are put into the Appendix.

\section{Setup}\label{setup}
We consider a two-dimensional process $\{(X_{t}, Y_{t}), 0\leq t\leq T\}$, which is defined on an appropriate filtered probability space $\left(\Omega,\mathcal{F},(\mathcal{F}_{t})_{t\geq 0},\mathbb{P}\right)$ with the following form:
	\begin{equation*}\label{model}
		\left\{
		\begin{aligned}
			\mathrm{d}X_{t}&=b_{t}^{X}\mathrm{d}t+\sigma_{t}^{X}\mathrm{d}W_{t}^{X}+\int_{\mathbb R}\delta^X(t-, x)\tilde{\mu}^{X}(\mathrm{d}t, \mathrm{d}x), ~~X_0=x,\\
			\mathrm{d}Y_{t}&=b_{t}^{Y}\mathrm{d}t+\sigma_{t}^{Y}\mathrm{d}W_{t}^{Y}+\int_{\mathbb R}\delta^Y(t-, y)\tilde{\mu}^{Y}(\mathrm{d}t, \mathrm{d}y), ~~Y_0=y,
		\end{aligned}
		\right.
	\end{equation*}
where $b_{t}^{Z}$ and $\sigma_{t}^{Z}$, $Z=X, Y$ are adapted and locally bounded $\mathrm{c\grave{a}dl\grave{a}g}$ processes, $(W_{t}^{X}, W_{t}^{Y})$ is a two-dimensional Gaussian process with independent increments, zero mean, and the covariance matrix is
	$$
	\begin{bmatrix}
		t & \int_{0}^{t}\rho_{s}\mathrm{d}s \\
		\int_{0}^{t}\rho_{s}\mathrm{d}s & t
	\end{bmatrix}
	, \ \ \ 0\leq t\leq T, $$
	where the $\rho_{t}, 0\leq t\leq T$ is a deterministic function and take values in the interval $[-1,1]$. Actually, the marginal processes  $W_{t}^{X}, W_{t}^{Y}$ are standard Brownian motions. There exists a standard Brownian motion $W_{t}^{\star}$ independent of $W_{t}^{X}$, and we can rewrite $$\mathrm{d}W_{t}^{Y}=\rho_{t}\mathrm{d}W_{t}^{X}+\sqrt{1-\rho_{t}^2}\mathrm{d}W_{t}^{\star}, 0\leq t\leq T.$$ Besides, $\tilde{\mu}^{X}$ and $\tilde{\mu}^{Y}$ are homogeneous poisson random measures with compensators $\nu^{X}(\mathrm{d}x)\otimes\mathrm{d}t$ and $\nu^{Y}(\mathrm{d}y)\otimes\mathrm{d}t$ respectively, and $\delta^{X}(t-, x)$ and $\delta^{Y}(t-, y)$ are prediction processes on $\mathbb{R}^{+}\times\mathbb{R}$ satisfying integrability conditions.
	
Due to technical reasons, we need two assumptions, which are similar to the work of \cite{todorov2012realized}. The Assumption \ref{asu1} is used to restrict the discontinuous part (jumps), and the Assumption \ref{asu2} is about the coefficient of the continuous part of price processes.
	\begin{asu}\label{asu1}
		The $L\acute{e}vy$ measure of $\tilde{\mu}^Z$, $Z=X,Y$ satisfy\\
		$$\mathbb{E}\left(\int_{0}^{t}\int_{\mathbb{R}}\left(|\delta^{Z}(s,z)|^{p}\vee|\delta^{Z}(s,z)|\right)\nu^{Z}(\mathrm{d}z)\mathrm{d}s\right)<\infty$$ for every $t>0$ and every $p\in(\beta,1)$, where $0\leq\beta<1$ is some constant.
	\end{asu}
	
The Assumption \ref{asu1} requires the jump components of price processes to be of finite variation, and the first moment of the jump processes exists. This condition is needed to show the asymptotic normality.	
	
	\begin{asu}\label{asu2}
		Assume that $\sigma_{t}^Z$, $Z=X,Y$, is an It$\hat{\text {o}}$ semi-martingale given by
		\begin{equation*}
			\sigma_{t}^{Z}=\sigma_{0}^{Z}+\int_{0}^{t}\tilde{b}_{s}^{Z}\mathrm{d}s+\int_{0}^{t}v_{s}^{Z}\mathrm{d}W_{s}^{Z}+\int_{0}^{t}v_{s}^{'Z}\mathrm{d}W_{s}^{'Z}+\int_{0}^{t}\int_{\mathbb{R}}\delta^{'Z}(s-,z)\tilde{\mu}^{'Z}(\mathrm{d}s,\mathrm{d}z),\\
		\end{equation*}
		where $W_{s}^{'Z}$ is a Brownian motion that is independent $W_{s}^{Z}$, $\tilde{\mu}^{'Z}$ is a homogenous Possion measure with $L\acute{e}vy$ measure $\nu^{'Z}(\mathrm{d}z)\otimes\mathrm{d}t$, which has arbitrary dependence with $\tilde{\mu}^{Z}$, and $\delta^{'Z}(t,z):\mathbb{R}^{+}\times \mathbb{R}\rightarrow\mathbb{R}$ is $\mathrm{c\grave{a}dl\grave{a}g}$ in $t$. We have for every $t$ and $s$ and some $\iota>0$,
		\begin{align*}
			&\mathbb{E}\left(|b_{t}^{Z}|^{3+\iota}+|\tilde{b}_{t}^{Z}|^{2}+|\sigma_{t}^{Z}|^{2}+|v_{t}^{Z}|^{3+\iota}+|v_{t}^{'Z}|^{3+\iota}+\int_{\mathbb{R}}|\delta^{'Z}(t,z)|^{3+\iota}\nu^{'Z}(\mathrm{d}z)\right)<C,\\
			&\mathbb{E}\left(|b_{t}^{Z}-b_{s}^{Z}|^{2}+|v_{t}^{Z}-v_{s}^{Z}|^{2}+|v_{t}^{'Z}-v_{s}^{'Z}|^{2}+\int_{\mathbb{R}}(\delta^{'Z}(t,z)-\delta^{'Z}(s,z))^{2}\nu^{'Z}(\mathrm{d}z)\right)<C|t-s|,
		\end{align*}
		where $C>0$ is some constant that does not depend on $t$ and $s$.
	\end{asu}
	
Assumption \ref{asu2} imposes integrability conditions on $b_{t}^{Z}$ and $\sigma_{t}^{Z}$, $Z=X, Y$ and bound their fluctuations over a short period of time. Moreover, we can also treat $b_{t}^{Z}$ and $\sigma_{t}^{Z}$, $Z=X, Y$ ``locally" as constants when the sampling frequency is high enough. This restriction is mild and widely used in the models of financial econometrics.

Next, we present our estimators. For a fixed time interval $[0,T]$, we observe $\{(X_{t}, Y_{t})\}$ at discrete time points $t_i^n=i\Delta_n$. When $\Delta_n\rightarrow 0$, we obtain the high frequency data $\{(X_{t_i^n}, Y_{t_i^n})\}$. The quantity of interest is the empirical joint  Laplace transform (integrated over the interval $[0, T]$) of $\sigma^{X}$ and $\sigma^{Y}$:
$$\int_0^T\mbox{e}^{-\langle (u, v), ((\sigma_{s}^{X})^2, (\sigma_{s}^{Y})^2)\rangle}\mathrm{d}s,$$
for $(u, v)\in \mathbb{R}_+^2$.
	Let $\Delta_i^nZ:=Z_{t_i^n}-Z_{t_{i-1}^n}$ for $Z=X, Y$,
	and $$\xi_i^n(u,v)=\cos\left(\frac{\sqrt{2u}\Delta_i^n{X}+\sqrt{2v}\Delta_{i+1}^n{Y}}{\sqrt{\Delta_n}}\right).$$
	We first propose an estimator with non-overlapped increments:
	\begin{equation*}
		V_n(u,v)=2\Delta_n\sum_{i=1}^{\lfloor n/2\rfloor}\xi_{2i-1}^n(u,v),
	\end{equation*}
	where $n=\lfloor T/\Delta_n \rfloor$. Our second estimator with overlapped increments is defined as:
	\begin{equation*}
		U_n(u,v)=\Delta_n\sum_{i=1}^{n-1}\xi_{i}^n(u,v).
	\end{equation*}
	
	\section{Asymptotic results}\label{asymptotic}
	In this section, we first show the consistency of these two estimators. Let $\sigma_i^Z:=\sigma_{t_i^n-}^Z$, $\Delta_i^n{W}^{Z}:=W_{t_i^n}^Z-W_{t_{i-1}^n}^Z$, $Z=X,Y$, ${\cal F}_i={\cal F}_{t_{i}^n}$. For the estimator $V_n(u,v)$, we have the following approximate
	$$\cos\left(\frac{\sqrt{2u}\Delta_{2i-1}^n{X}+\sqrt{2v}\Delta_{2i}^n{Y}}{\sqrt{\Delta_n}}\right) \approx \cos\left(\frac{\sqrt{2u}\sigma_{2i-2}^{X}\Delta_{2i-1}^n{W}^{X}+\sqrt{2v}\sigma_{2i-2}^{Y}\Delta_{2i}^n{W}^{Y}}{\sqrt{\Delta_n}}\right),$$
	where $i=1,\cdots,\lfloor n/2\rfloor$. In the Appendix, we will show that the error between the two terms is asymptotically negligible. Moreover,
	\begin{align*}
		&\mathbb{E}\left[\cos\left(\frac{\sqrt{2u}\sigma_{2i-2}^{X}\Delta_{2i-1}^n{W}^{X}+\sqrt{2v}\sigma_{2i-2}^{Y}\Delta_{2i}^n{W}^{Y}}{\sqrt{\Delta_n}}\right)\Big|\mathcal{F}_{2i-2}\right]\\
		=&\mathbb{E}\left[\cos\left(\frac{\sqrt{2u}\sigma_{2i-2}^{X}\Delta_{2i-1}^n{W}^{X}}{\sqrt{\Delta_{n}}}\right)\cos\left(\frac{\sqrt{2v}\sigma_{2i-2}^{Y}\Delta_{2i}^n{W}^{Y}}{\sqrt{\Delta_{n}}}\right)\Big|\mathcal{F}_{2i-2}\right]\\
		=&e^{-\langle(u,v), ((\sigma_{2i-2}^{X})^2,(\sigma_{2i-2}^{Y})^2)\rangle}.
	\end{align*}
	The equalities hold because $\mathrm{sin}(\cdot)$ is an odd function and $\Delta_{2i-1}^n{W}^{X}$ is independent of $\Delta_{2i}^n{W}^{Y}$.
	According to the weak law of large numbers,
	\begin{equation}\label{consistency}
		2\Delta_{n}\sum_{i=1}^{\lfloor n/2\rfloor}\left\{\cos\left(\frac{\sqrt{2u}\sigma_{2i-2}^{X}\Delta_{2i-1}^n{W}^{X}+\sqrt{2v}\sigma_{2i-2}^{Y}\Delta_{2i}^n{W}^{Y}}{\sqrt{\Delta_n}}\right)-e^{-\langle(u,v), ((\sigma_{2i-2}^{X})^2,(\sigma_{2i-2}^{Y})^2)\rangle}\right\}\stackrel{P}{\longrightarrow}0.
	\end{equation}
	The term $2\Delta_{n}\sum_{i=1}^{\lfloor n/2\rfloor}e^{-\langle(u,v), ((\sigma_{2i-2}^{X})^2,(\sigma_{2i-2}^{Y})^2)\rangle}$ converges to $\int_0^T\mbox{e}^{-\langle (u, v), ((\sigma_{s}^{X})^2, (\sigma_{s}^{Y})^2)\rangle}\mathrm{d}s$. The consistency of $U_{n}$ can be obtained in the same way.
	
	Next, we will present the central limit theorems of the estimators $V_{n}$ and $U_{n}$. The notation $\stackrel{\cal S}{\longrightarrow}$ represents the stable convergence in law in the Theorems \ref{thm1} and \ref{thm2}. The stable convergence in law is stronger than the convergence in law and weaker than the convergence in probability. More details about this convergence mode can be found in \cite{renyi1963stable}, \cite{aldous1978mixing}, \cite{jacod2012discretization}
	and \cite{ait2014high}. We have the following functional central limit theorem.
	\begin{thm}\label{thm1} Under Assumptions \ref{asu1} and \ref{asu2} , we have
		
		\begin{equation}\label{equ3}
			\frac{1}{\sqrt{\Delta_n}}\left(V_n(u,v)-\int_0^T\mbox{e}^{-\langle (u, v), ((\sigma_{s}^{X})^2, (\sigma_{s}^{Y})^2)\rangle}\mathrm{d}s\right)\stackrel{\cal S}{\longrightarrow} \Phi_T(u,v),
		\end{equation}
		where the convergence is on the space ${\cal C}(\mathbb{R}_+^2)$ of continuous functions indexed by $(u, v)$ and equipped with the local uniform topology {\rm (}i.e., uniformly over compact sets of $(u, v)\in \mathbb{R}_+^2${\rm )}, and the process $\Phi_T(u,v)$ is defined on an extension of the original probability space and is an ${\cal F}$-conditionally Gaussian process with zero-mean function and covariance function $\int_0^T F(\sqrt{u}\sigma_{s}^{X}, \sqrt{v}\sigma_{s}^{Y}, \sqrt{u'}\sigma_{s}^{X}, \sqrt{v'}\sigma_{s}^{Y})\mathrm{d}s$ for every $(u, v, u', v')\in \mathbb{R}_+^4$ with
		\begin{equation*}
			F(x,y,\bar{x},\bar{y})=\mbox{e}^{-(x^2+y^2+\bar{x}^2+\bar{y}^2)}\cdot\left\{\mbox{e}^{-2(x\bar{x}+y\bar{y})}+\mbox{e}^{2(x\bar{x}+y\bar{y})}-2\right\}.
		\end{equation*}
		
		\begin{ignore}{
				A consistent estimator for the covariance function of $\Phi_T(u,v)$ is given by
				\begin{eqnarray}\label{estvariance}
					\hat{\Gamma}_n&=&\frac{\Delta_n}{2}\sum_{i=1}^{n-1}\xi_i^n(u+u'+\sqrt{uu'}, v+v'+\sqrt{vv'})\\
					&&+\frac{\Delta_n}{2}\sum_{i=1}^{n-1}\xi_i^n(u+u'-\sqrt{uu'}, v+v'-\sqrt{vv'})\\
					&&+\Delta_n\sum_{i=1}^{n-1}\cos(\frac{\sqrt{2u}\Delta_i^n{X_1}+\sqrt{2v'}\Delta_{i+1}^n{X_2}}{\sqrt{\Delta_n}})\cdot
					\Delta_n\sum_{i=1}^{n}\cos(\frac{\sqrt{2u'}\Delta_i^n{X_1}}{\sqrt{\Delta_n}}\big)\cos(\frac{\sqrt{2v}\Delta_i^n{X_2}}{\sqrt{\Delta_n}})\nonumber\\
					&&+\Delta_n\sum_{i=1}^{n-1}\cos(\frac{\sqrt{2u'}\Delta_i^n{X_1}+\sqrt{2v}\Delta_{i+1}^n{X_2}}{\sqrt{\Delta_n}})\cdot
					\Delta_n\sum_{i=1}^{n}\cos(\frac{\sqrt{2u}\Delta_i^n{X_1}}{\sqrt{\Delta_n}}\big)\cos(\frac{\sqrt{2v'}\Delta_i^n{X_2}}{\sqrt{\Delta_n}})\nonumber\\
					&&-3\Delta_n\sum_{i=1}^{n-1}\cos(\frac{\sqrt{2(u+u')}\Delta_i^n{X_1}+\sqrt{2(v+v')}\Delta_{i+1}^n{X_2}}{\sqrt{\Delta_n}}).
				\end{eqnarray}
				
				\begin{eqnarray}
					\tilde{\Gamma}_n&=&\Delta_n\sum_{i=1}^{\lfloor n/2\rfloor-1}\big((\xi_{2i-1}^n(u,v)-\xi_{2i+1}^n(u,v)\big)\big(\xi_{2i-1}^n(u',v')-\xi_{2i+1}^n(u',v')\big).
				\end{eqnarray}
			}
		\end{ignore}
		
		A consistent estimator for the covariance function of $\Phi_T(u,v)$ is given by
		\begin{align*}
			{\tilde{\Gamma}_n}=&{4\Delta_{n}\sum_{i=1}^{\lfloor n/2\rfloor}\mathrm{cos}\left(\frac{\sqrt{2u}\Delta_{2i-1}^{n}X+\sqrt{2v}\Delta_{2i}^{n}Y}{\sqrt{\Delta_{n}}}\right)\mathrm{cos}\left(\frac{\sqrt{2u'}\Delta_{2i-1}^{n}X+\sqrt{2v'}\Delta_{2i}^{n}Y}{\sqrt{\Delta_{n}}}\right)}\nonumber\\
			&{-4\Delta_{n}\sum_{i=1}^{\lfloor n/2\rfloor}\mathrm{cos}\left(\frac{\sqrt{2(u+u')}\Delta_{2i-1}^{n}X+\sqrt{2(v+v')}\Delta_{2i}^{n}Y}{\sqrt{\Delta_{n}}}\right)}.
		\end{align*}
	\end{thm}
	
The above estimator uses the non-overlapped increments, which yields a neat asymptotic covariance function. However, the estimation efficiency can be improved by using the overlapped increments:
		\begin{equation*}
			U_n(u,v)=\Delta_n\sum_{i=1}^{n-1}\xi_i^n(u,v).
		\end{equation*}
	We have the following results.
	\begin{thm}\label{thm2} Under Assumptions \ref{asu1} and \ref{asu2} , we have
		\begin{equation}\label{conve2}
			\frac{1}{\sqrt{\Delta_n}}\left(U_n(u,v)-\int_0^{T}\mbox{e}^{-\langle (u, v), ((\sigma_{s}^{X})^2, (\sigma_{s}^{Y})^2)\rangle}\mathrm{d}s\right)\stackrel{\cal S}{\longrightarrow} \Phi_T(u,v),
		\end{equation}
where the convergence is on the space ${\cal C}(\mathbb{R}_+^2)$ continuous functions indexed by $(u, v)$ and equipped with the local uniform topology {\rm (}i.e., uniformly over compact sets of $(u, v)\in \mathbb{R}_+^2${\rm)}, and the process $\Phi_T(u,v)$ is defined on an extension of the original probability space and is an ${\cal F}$-conditionally Gaussian process with zero-mean function and covariance function $\int_0^TF(\sqrt{u}\sigma_{s}^{X}, \sqrt{v}\sigma_{s}^{Y}, \sqrt{u'}\sigma_{s}^{X}, \sqrt{v'}\sigma_{s}^{Y};\rho_s)\mathrm{d}s$ for every $(u, v, u', v')\in \mathbb{R}_+^4$ with
		\begin{equation*}
			F(x, y,\bar{x}, \bar{y};z)=\frac{1}{2}\mbox{e}^{-(x^2+y^2+\bar{x}^2+\bar{y}^2)}\cdot\left\{\frac{1+\mbox{e}^{{4}x\bar{x}+{4}y\bar{y}}}
			{\mbox{e}^{{2}x\bar{x}+{2}y\bar{y}}}+
			\frac{1+\mbox{e}^{4\bar{x}yz}}{\mbox{e}^{2\bar{x}yz}}+
			\frac{1+\mbox{e}^{4x\bar{y}z}}{\mbox{e}^{2x\bar{y}z}}-6\right\}.
		\end{equation*}
		\begin{ignore}{
				A consistent estimator for the covariance function of $\Phi_T(u,v)$ is given by
				\begin{eqnarray}\label{estvariance}
					\hat{\Gamma}_n&=&\frac{\Delta_n}{2}\sum_{i=1}^{n-1}\xi_i^n(u+u'+\sqrt{uu'}, v+v'+\sqrt{vv'})\\
					&&+\frac{\Delta_n}{2}\sum_{i=1}^{n-1}\xi_i^n(u+u'-\sqrt{uu'}, v+v'-\sqrt{vv'})\\
					&&+\Delta_n\sum_{i=1}^{n-1}\cos(\frac{\sqrt{2u}\Delta_i^n{X_1}+\sqrt{2v'}\Delta_{i+1}^n{X_2}}{\sqrt{\Delta_n}})\cdot
					\Delta_n\sum_{i=1}^{n}\cos(\frac{\sqrt{2u'}\Delta_i^n{X_1}}{\sqrt{\Delta_n}}\big)\cos(\frac{\sqrt{2v}\Delta_i^n{X_2}}{\sqrt{\Delta_n}})\nonumber\\
					&&+\Delta_n\sum_{i=1}^{n-1}\cos(\frac{\sqrt{2u'}\Delta_i^n{X_1}+\sqrt{2v}\Delta_{i+1}^n{X_2}}{\sqrt{\Delta_n}})\cdot
					\Delta_n\sum_{i=1}^{n}\cos(\frac{\sqrt{2u}\Delta_i^n{X_1}}{\sqrt{\Delta_n}}\big)\cos(\frac{\sqrt{2v'}\Delta_i^n{X_2}}{\sqrt{\Delta_n}})\nonumber\\
					&&-3\Delta_n\sum_{i=1}^{n-1}\cos(\frac{\sqrt{2(u+u')}\Delta_i^n{X_1}+\sqrt{2(v+v')}\Delta_{i+1}^n{X_2}}{\sqrt{\Delta_n}}).
				\end{eqnarray}
		}\end{ignore}
	
		A consistent estimator of the asymptotic covariance function is given by
		\begin{align*}
			{\tilde{\Gamma}_n}=&\Delta_{n}\sum_{i=1}^{n-1}\mathrm{cos}\left(\frac{\sqrt{2u}\Delta_{i}^{n}X+\sqrt{2v}\Delta_{i+1}^{n}Y}{\sqrt{\Delta_{n}}}\right)\mathrm{cos}\left(\frac{\sqrt{2u'}\Delta_{i}^{n}X+\sqrt{2v'}\Delta_{i+1}^{n}Y}{\sqrt{\Delta_{n}}}\right)\nonumber\\ &+\Delta_{n}\sum_{i=1}^{n-2}\mathrm{cos}\left(\frac{\sqrt{2u}\Delta_{i-1}^{n}X}{\sqrt{\Delta_{n}}}\right) \mathrm{cos}\left(\frac{\sqrt{2v}\Delta_{i+1}^{n}Y}{\sqrt{\Delta_{n}}}\right)\mathrm{cos}\left(\frac{\sqrt{2u'}\Delta_{i+1}^{n}X}{\sqrt{\Delta_{n}}}\right)\mathrm{cos}\left(\frac{\sqrt{2v'}\Delta_{i+2}^{n}Y}{\sqrt{\Delta_{n}}}\right)\nonumber\\
			&+\Delta_{n}\sum_{i=1}^{n-2}\mathrm{cos}\left(\frac{\sqrt{2u'}\Delta_{i-1}^{n}X}{\sqrt{\Delta_{n}}}\right) \mathrm{cos}\left(\frac{\sqrt{2v'}\Delta_{i+1}^{n}Y}{\sqrt{\Delta_{n}}}\right)\mathrm{cos}\left(\frac{\sqrt{2u}\Delta_{i+1}^{n}X}{\sqrt{\Delta_{n}}}\right)\mathrm{cos}\left(\frac{\sqrt{2v}\Delta_{i+2}^{n}Y}{\sqrt{\Delta_{n}}}\right)\nonumber\\
			&-3\Delta_{n}\sum_{i=1}^{n-1}\mathrm{cos}\left(\frac{\sqrt{2(u+u')}\Delta_{i}^{n}X+\sqrt{2(v+v')}\Delta_{i+1}^{n}Y}{\sqrt{\Delta_{n}}}\right).
		\end{align*}
	\end{thm}

\begin{rem}$U_{n}(u,v)$ makes good use of the data compared to $V_{n}(u,v)$ and naturly have smaller asymptotic variance. Indeed, some straightforward computation can show this, and here we display the plots of two asymptotic variances as a function of $(u,v)$
\ignore{
The following Proposition \ref{pro1} reveals the relationship between the asymptotic variance of $V_{n}(u,v)$ and  $U_{n}(u,v)$.
\begin{pro}\label{pro1}
The asymptotic variance of $U_{n}(u,v)$ is smaller than $V_{n}(u,n)$.
\end{pro}
\begin{proof}
We need to prove
\begin{eqnarray*}
&&e^{-2(u\sigma_{x}^2+v\sigma_{y}^2)}\cdot\left\{e^{-2(u\sigma_{x}^2+v\sigma_{y}^2)}+e^{2(u\sigma_{x}^2+v\sigma_{y}^2)}-2\right\}\\
&>&\frac{1}{2}e^{-2(u\sigma_{x}^2+v\sigma_{y}^2)}\cdot\left\{e^{-2(u\sigma_{x}^2+v\sigma_{y}^2)}+e^{2(u\sigma_{x}^2+v\sigma_{y}^2)}+2\left(e^{-2\sqrt{uv}\sigma_{x}\sigma_{y}\rho}+e^{2\sqrt{uv}\sigma_{x}\sigma_{y}\rho}\right)-6\right\}
\end{eqnarray*}
For convenience of expression, we let,
\begin{eqnarray*}
  a:&=& e^{-2(u\sigma_{x}^2+v\sigma_{y}^2)}+e^{2(u\sigma_{x}^2+v\sigma_{y}^2)}-2 ,\\
  b:&=& 2\left(e^{-2\sqrt{uv}\sigma_{x}\sigma_{y}\rho}+e^{2\sqrt{uv}\sigma_{x}\sigma_{y}\rho}-4\right).
\end{eqnarray*}
It suffices to prove $a>b$. Further, we define $c:=u\sigma_{x}^{2}+v\sigma_{y}^2$, $d:=2\sqrt{uv}\sigma_{x}\sigma_{y}$, $c,d\in\mathbb{R}^{+}$. It is easy to find $c\geq d$, $|d\rho|\leq |d|$, and
\begin{eqnarray*}
  a&=&e^{-2c}+e^{2c}-2\\
   &\geq&e^{-2d}+e^{2d}-2\\
   &>&2(e^{-d}+e^{d})-4\\
   &\geq&2(e^{-d\rho}+e^{d\rho})-4=b\\
\end{eqnarray*}
\end{proof}
\end{rem}
}
in Figure \ref{figure0}. In the figure, we let $T=1$, $\sigma_{s}^{X}=1$, $\sigma_{s}^{Y}=1$, $u=u'$, $v=v'$ and $\rho=0.5$. From the Figure \ref{figure0}, we see that the asymptotic variance of $U_n(u,v)$ is uniformly smaller than that of $V_n(u,v)$.
\begin{figure}[htb]
		\centering
	\includegraphics[width=7cm,height=5.5cm]{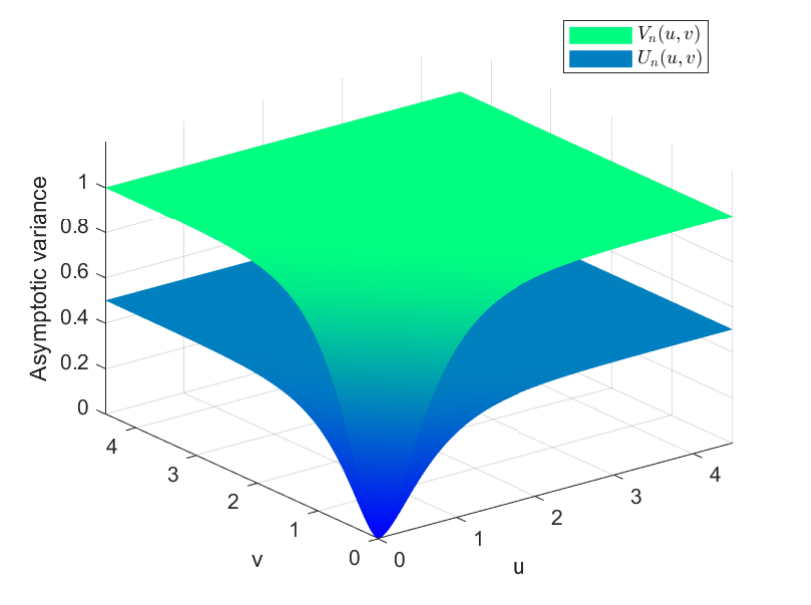}~~~~\includegraphics[width=7cm,height=5.5cm]{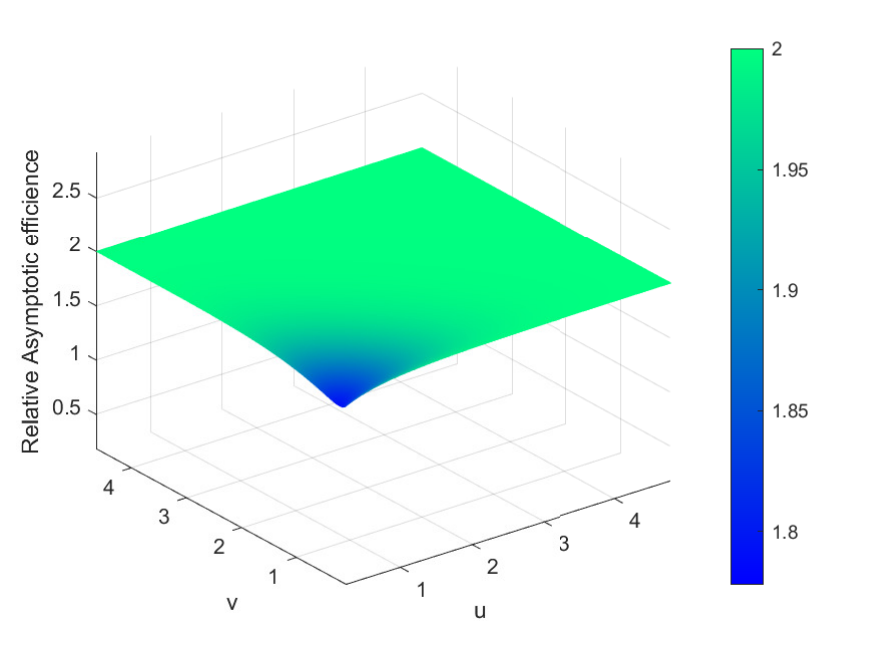}
\label{figure0}	
\caption{The asymptotic variance of $V_n(u,v)$ and $U_n(u,v)$}	
\end{figure}
\end{rem}

Although $U_{n}(u,v)$ is better than $V_{n}(u,v)$ in the sense of estimation performance, we still keep $V_{n}(u,v)$ in the paper for the following reasons. First, it is seen that the asymptotic variance of $V_{n}(u,v)$ does not depend on the correlation process $\rho_t$ of $X$ and $Y$. Thus, it is robust in the dependence between two latent processes. Second, we can actually get an improved version of $V_{n}(u,v)$ by adding the other symmetric part, namely,
    $$V_{n}'(u,v)=\Delta_{n}\sum_{i=1}^{\lfloor n/2\rfloor}\left[\cos\left(\frac{\sqrt{2u}\Delta_{2i-1}^{n}X+\sqrt{2v}\Delta_{2i}^{n}Y}{\sqrt{\Delta_{n}}}\right)+\cos\left(\frac{\sqrt{2v}\Delta_{2i-1}^{n}Y+\sqrt{2u}\Delta_{2i}^{n}X}{\sqrt{\Delta_{n}}}\right)\right].$$
We can get the central limit theory of $V'_{n}(u,v)$ as the following:
    \begin{equation}
			\frac{1}{\sqrt{\Delta_n}}\left(V'_n(u,v)-\int_0^{T}\mbox{e}^{-\langle (u, v), ((\sigma_{s}^{X})^2, (\sigma_{s}^{Y})^2)\rangle}\mathrm{d}s\right)\stackrel{\cal S}{\longrightarrow} \Phi_T(u,v),
		\end{equation}
where the process $\Phi_T(u,v)$ is the same as the limit process of $U_{n}(u,v)$. We also include this estimator in the simulation study in Section \ref{simulation}. However, using a similar modification to the estimator $U_{n}(u,v)$ does not improve the estimation efficiency.

\ignore{
We can use the same technology on estimator $U_{n}(u,v)$, and achieve the following estimator:
$$U'_{n}(u,v)=\Delta_{n}\sum_{i=1}^{n-1}\frac{1}{2}\left[\cos\left(\frac{\sqrt{2u}\Delta_{i}^{n}X+\sqrt{2v}\Delta_{i+1}^{n}Y}{\sqrt{\Delta_{n}}}\right)+\cos\left(\frac{\sqrt{2v}\Delta_{i}^{n}Y+\sqrt{2u}\Delta_{i+1}^{n}X}{\sqrt{\Delta_{n}}}\right)\right].$$
It is very cumbersome to calculate the asymptotic covariance function of the limit process. However, we list the simulation results in Section \ref{simulation} and compare them to the other three estimators.
}

\section{Joint Infill and Long-time Span Asymptotics}\label{longspan}
 In this section, we extend the above analysis (for fixed $T$) and consider the asymptotic results for the stationary process when $T\to\infty$ jointly $\Delta_{n}\to 0$. To facilitate the derivation of this theory, we require some additional notations. We define,
 \begin{align*}
 Z_{t}^{x,y}(u,v)&=\int_{t-1}^{t}e^{-u(\sigma_{s}^{x})^2-v(\sigma_{s}^{y})^2}{\rm{d}}s,
 &\hat{Z}_{t}^{x,y}(u,v)&=\sum\limits_{\lfloor(t-1)/\Delta_{n}\rfloor+1}^{\lfloor t/\Delta_{n}\rfloor}\Delta_{n}\cos\left(\frac{\sqrt{2u}\Delta_i^n{X}+\sqrt{2v}\Delta_{i+1}^n{Y}}{\sqrt{\Delta_n}}\right),\\
 Z_{t}^{x}(u)&=\int_{t-1}^{t}e^{-u(\sigma_{s}^{x})^2}ds,
 &\hat{Z}_{t}^{x}(u)&=\sum\limits_{\lfloor(t-1)/\Delta_{n}\rfloor+1}^{\lfloor t/\Delta_{n}\rfloor}\Delta_{n}\cos\left(\frac{\sqrt{2u}\Delta_i^n{X}}{\sqrt{\Delta_n}}\right),\\
 Z_{t}^{y}(v)&=\int_{t-1}^{t}e^{-u(\sigma_{s}^{y})^2}ds,
 &\hat{Z}_{t}^{y}(v)&=\sum\limits_{\lfloor(t-1)/\Delta_{n}\rfloor+1}^{\lfloor t/\Delta_{n}\rfloor}\Delta_{n}\cos\left(\frac{\sqrt{2u}\Delta_i^n{Y}}{\sqrt{\Delta_n}}\right),
\end{align*}
and $\mu_{t}^{x,y}(u,v):=\mathbb{E}[Z_{t}^{x,y}(u,v)]$, $\mu_{t}^{x}(u):=\mathbb{E}[Z_{t}^{x}(u)]$, $\mu_{t}^{y}(v):=\mathbb{E}[Z_{t}^{y}(v)]$. Due to the stationarity, we have
$$\mu_{t}^{x,y}(u,v)=\mu_{1}^{x,y}(u,v), \mu_{t}^{x}(u)=\mu_{1}^{x}(u), \mu_{t}^{y}(v)=\mu_{1}^{y}(v).$$
The key result is a central limit theorem of sample averages of the Laplace transform of volatilities, based on which we can estimate $\mu_{t}^{x,y}$, $\mu_{t}^{x}$ and $\mu_{t}^{y}$. Secondly, we can test whether $\sigma_{s}^x$ and $\sigma_{s}^y$ are dependent on each other.
We need the following assumption to derive the long-time span asymptotic behavior.

\begin{asu}\label{asu3}
The volatility processes $\sigma_{t}^{z}$, $z=x,y$ are stationary and $\alpha$-mixing processes with coefficient $\alpha^{mix}=O(t^{-\gamma})$ for some  $\gamma>1$ when $t\to \infty$, where
$$\alpha_{t}^{mix}=\underset{A\in {\mathcal{F}}_{0},B\in {\mathcal{F}}^{t}}{\rm sup}|\mathbb{P}(A\cap B)-\mathbb{P}(A)\mathbb{P}(B)|,$$
$${\mathcal{F}}_{0}=\sigma(\sigma_{s}^{z}, { W_{s}^{z}}, z=x,y, s\leq 0) \quad and\quad {\mathcal{F}}^{t}=\sigma(\sigma_{s}^{z}, { W_{s}^{z}-W_{t}^{z}}, z=x,y, s\geq t).$$
\end{asu}	

The stationary and mixing condition in Assumption \ref{asu3} is required to ensure the central limit theorem exists under the long-span setting. In short, we hope to restrict the memory of the volatilities so that the volatility processes are not ``too'' strongly dependent. Compared with mixing conditions of \cite{todorov2012realized}, the Assumption \ref{asu3} is slightly weaker and not restrictive in practice as it can capture a wide variety of volatility models, such as a large class of processes driven by Brownian motion(e.g., \cite{heston1993closed}) or the L\'{e}vy-driven Ornstein-Uhlenbeck model of \cite{barndorff2001non}, where volatility is governed by general (positive) processes.

\begin{thm}\label{thm3}
Under Assumption \ref{asu1}--\ref{asu3}, if $T\to\infty$, $\Delta_{n}\to 0$ and $\sqrt{T\Delta_{n}}\to 0$, we have
$$\sqrt{T}\begin{pmatrix}
\frac{1}{T}\sum\limits_{t=1}^{T}\hat{Z}_{t}^{x,y}(u,v)-\mu_{1}^{x,y}(u,v)\\
\frac{1}{T}\sum\limits_{t=1}^{T}\hat{Z}_{t}^{x}(u)-\mu_{1}^{x}(u)\\
\frac{1}{T}\sum\limits_{t=1}^{T}\hat{Z}_{t}^{y}(v)-\mu_{1}^{y}(v)\\
\end{pmatrix}
\stackrel{\cal S}{\longrightarrow} {\bf \Phi}(u,v),$$
where the convergence is on the space ${\cal C}(\mathbb{R}_+^2)$ of continuous functions indexed by $(u, v)$ and equipped with the local uniform topology {\rm (}i.e., uniformly over compact sets of $(u, v)\in \mathbb{R}_+^2${\rm )}, and the process ${\bf \Phi}(u,v)$ is defined on an extension of the original probability space and is an ${\cal F}$-conditionally Gaussian process with zero-mean function and covariance function matrices $V([u,v],[u',v'])$ for every $(u, v, u', v')\in \mathbb{R}_+^4$ with
\begin{equation*}
V_{ab}([u,v],[u',v'])=\mathbb{E}[A_{a}(0)A'_{b}(0)]+2\sum_{l=1}^{\infty}\mathbb{E}[A_{a}(0)A_{b}'(l)],
\end{equation*}
here $a,b=1,2,3, a\leq b$ and,
\begin{align*}
A_{1}(l)&=Z_{l+1}^{x,y}(u,v)-\mu_{1}^{x,y}(u,v),& A_{1}'(l)&=Z_{l+1}^{x,y}(u',v')-\mu_{1}^{x,y}(u',v'),\\
A_{2}(l)&=Z_{l+1}^{x}(u)-\mu_{1}^{x}(u),&  A_{2}'(l)&=Z_{l+1}^{x}(u')-\mu_{1}^{x}(u'),\\
A_{3}(l)&=Z_{l+1}^{y}(v)-\mu_{1}^{y}(u),&  A_{3}'(l)&=Z_{l+1}^{y}(v')-\mu_{1}^{y}(v').
\end{align*}
Moreover, if $L_{T}$ is a deterministic sequence of integers satisfying $\frac{L_{T}}{\sqrt{T}}\to 0$ and $L_{T}\Delta_{n}^{1/2}\to 0$ as $T\to \infty$, $\Delta_{n}\to 0$, a consistent estimator of the asymptotic covariance function matrices is given by $\hat{V}([u,v],[u', v'])$ for every $(u, v, u', v')\in \mathbb{R}_+^4$ with
\begin{equation*}
\hat{V}_{ab}([u,v],[u', v'])=\hat{C}_{ab}^{0}([u,v],[u', v'])+\sum_{i=1}^{L_{T}}\omega(i,L_{T})(\hat{C}_{ab}^{i}([u,v],[u',v'])+\hat{C}_{ab}^{i}([u',v'],[u,v])),
\end{equation*}
where, $a,b=1,2,3, a\leq b$. The $\omega(i, L_{T})$ is either a Bartlett or Parzen Kernel. In addition,
\begin{equation*}
\hat{C}_{ab}^{l}([u,v],[u',v'])=\frac{1}{T}\sum_{t=l+1}^{T}B_{at}(0)B_{bt}'(l),
\end{equation*}
with
\begin{align*}
B_{1t}(l)&=\hat{Z}_{t-l}^{x,y}(u,v)-\frac{1}{T}\sum_{t=1}^{T}\hat{Z}^{x,y}(u,v),& B_{1t}'(l)&=\hat{Z}_{t-l}^{x,y}(u',v')-\frac{1}{T}\sum_{t=1}^{T}\hat{Z}^{x,y}(u',v'),\\
B_{2t}(l)&=\hat{Z}_{t-l}^{x}(u)-\frac{1}{T}\sum_{t=1}^{T}\hat{Z}^{x}(u),& B_{2t}'(l)&=\hat{Z}_{t-l}^{x}(u')-\frac{1}{T}\sum_{t=1}^{T}\hat{Z}^{x}(u'),\\
B_{3t}(l)&=\hat{Z}_{t-l}^{y}(v)-\frac{1}{T}\sum_{t=1}^{T}\hat{Z}^{y}(v),& B_{3t}'(l)&=\hat{Z}_{t-l}^{y}(v')-\frac{1}{T}\sum_{t=1}^{T}\hat{Z}^{y}(v').
\end{align*}
\end{thm}

Finally, an application of the Delta method yields the following result.

\begin{corollary}\label{cor1}
Under the same conditions and notations as Theorem \ref{thm3}, we have
\begin{equation*}
\sqrt{T}\left(\frac{1}{T}\sum_{t=1}^{T}\hat{Z}_{t}^{x,y}(u,v)-\frac{1}{T}\sum_{t=1}^{T}\hat{Z}_{t}^{x}(u)\frac{1}{T}\sum_{t=1}^{T}\hat{Z}_{t}^{y}(v)-\Big(\mu_{1}^{x,y}(u,v)-\mu_{1}^{x}(u)\mu_{1}^{y}(v)\Big)\right)\stackrel{\cal S}\rightarrow {\boldsymbol{\gamma}}\cdot {\bf \Phi}(u,v),
\end{equation*}
where, ${\boldsymbol{\gamma}}=[1,-\mu_{1}^{y}(v),-\mu_{1}^{x}(u)]$.
\end{corollary}

This result can be used to test for the dependence of the volatilities. That is, if the volatilities are independent, then for all $(u,v)\in {\cal{R}}_{+}^2$, we have
$$
\mu_{1}^{x,y}(u,v)-\mu_{1}^{x}(u)\mu_{1}^{y}(v)=0,
$$
hence
\begin{equation*}
\hat{\cal S}_{T,n}=\sqrt{T}\left(\frac{1}{T}\sum_{t=1}^{T}\hat{Z}_{t}^{x,y}(u,v)-\frac{1}{T}\sum_{t=1}^{T}\hat{Z}_{t}^{x}(u)\frac{1}{T}\sum_{t=1}^{T}\hat{Z}_{t}^{y}(v)\right)\stackrel{\cal S}\rightarrow {\boldsymbol{\gamma}}\cdot {\bf \Phi}(u,v).
\end{equation*}
According to the Continuous mapping theorem,
$$
\iint[\hat{\cal S}_{T,n}(u,v)]^2{\rm{d}}u{\rm{d}}v \stackrel{\cal S}\rightarrow\iint[ {\boldsymbol{\gamma}} \cdot  {\boldsymbol \Phi}(u,v)]^2{\rm{d}}u{\rm{d}}v.
$$
Therefore, the critical region can be $C=\{||\hat{\cal S}_{T,n}||^2\geq d_{\alpha}$\} where $d_{\alpha}$ can be determined by the results of Theorem \ref{thm3} and the norm $||\cdot||$ is defined in the Hilbert space ${\cal{L}}^{2}$,
$${\cal{L}}^{2}=\left\{f:{\cal{R}}_{+}^{2}\rightarrow{\cal{R}}\Big|\iint_{{\cal{R}}_{+}^{2}}f(u,v)^2{\rm{d}}u{\rm{d}}v < \infty\right\}.$$

\section{Simulation studies}\label{simulation}
In this section, we conduct some simulation studies to examine the finite sample performance of our proposed estimators.
	\subsection{Fixed $T$}
	Throughout this simulation, we consider the fixed time interval with $T=1$ and observe the process at some discrete time points $\{i\Delta_{n},i=1,2,...,n\}$, $n=1760$, which correspond to monthly observations with 5-minutes frequency. We study the following three commonly used stochastic volatility models.
	\begin{ex}\label{ex1}
		We generate the high-frequency data $\{(X_{t}, Y_{t}), t\in [0, T]\}$ from the following processes:
		\begin{equation*}
			\left\{
			\begin{aligned}
				\mathrm{d}X_{t}&=0.03\mathrm{d}t+\sigma_{t}^{X}\mathrm{d}W_{t}^{X},\\
				\mathrm{d}Y_{t}&=0.04\mathrm{d}t+\sigma_{t}^{Y}\mathrm{d}W_{t}^{Y},
			\end{aligned}
			\right.
		\end{equation*}
where, $\mathrm{d}W_{t}^{Y}=\rho_{t}\mathrm{d}W_{t}^{X}+\sqrt{1-\rho_{t}^2}\mathrm{d}W_{t}^{\star}$, $W_{t}^{X}$ and $W_{t}^{\star}$ are two mutually independent Brownian motions. Without loss of generality, we let $\rho_{t}=0.5$ throughout the simulation study. Moreover, $\sigma_{t}^{X}=\mathrm{exp}\{0.3125-0.125\tau_{t}^{X}\}$, $\mathrm{d}\tau_{t}^{X}=-0.025\tau_{t}^{X}\mathrm{d}t+\mathrm{d}B_{t}^{X}$, and $\sigma_{t}^{Y}=\mathrm{exp}\{0.4500-0.325\tau_{t}^{Y}\}$, $\mathrm{d}\tau_{t}^{Y}=-0.030\tau_{t}^{Y}\mathrm{d}t+\mathrm{d}B_{t}^{Y}$.
		Here, $W_{t}^{Z}$ and $B_{t}^{Z}$, $Z=X,Y$ are independent of each other.
	\end{ex}
	
	\begin{ex}\label{ex2}
		$\{(X_{t}, Y_{t}), t\in[0, T]\}$ follows a stochastic volatility model with Poisson jumps:
		\begin{equation*}
			\left\{
			\begin{aligned}
				\mathrm{d}X_{t}&=0.03\mathrm{d}t+\sigma_{t}^{X}\mathrm{d}W_{t}^{X}+\displaystyle\sum_{i=1}^{N_{t}^{X}}Y_{i},\\
				\mathrm{d}Y_{t}&=0.04\mathrm{d}t+\sigma_{t}^{Y}\mathrm{d}W_{t}^{Y}+\displaystyle\sum_{i=1}^{N_{t}^{Y}}Y_{i},
			\end{aligned}
			\right.
		\end{equation*}
		where $N_{t}^{X}$ and $N_{t}^{Y}$ are Poisson processes with $\mathbb{E}[N_{t}^{X}]=2t$, $\mathbb{E}[N_{t}^{Y}]=3t$. Furthermore, we set the jump sizes $Y_{i}\stackrel{i.i.d}\sim N(0,1)$. The generation of $\sigma_{t}^{X}$, $\sigma_{t}^{Y}$, $\mathrm{d}W_{t}^{X}$ and $\mathrm{d}W_{t}^{Y}$ are the same as in Example \ref{ex1}.
	\end{ex}
	
	\begin{ex}\label{ex3}
		$\{(X_{t}, Y_{t}), t\in[0, T]\}$ follows a stochastic volatility model with $\alpha$-stable jumps:
		\begin{equation*}
			\left\{
			\begin{aligned}
				\mathrm{d}X_{t}&=0.03\mathrm{d}t+\sigma_{t}^{X}\mathrm{d}W_{t}^{X}+\mathrm{d}J_{t}^{X},\\
				\mathrm{d}Y_{t}&=0.04\mathrm{d}t+\sigma_{t}^{Y}\mathrm{d}W_{t}^{Y}+\mathrm{d}J_{t}^{Y},
			\end{aligned}
			\right.
		\end{equation*}
		where $\sigma_{t}^{X}$, $\sigma_{t}^{Y}$, $\mathrm{d}W_{t}^{X}$ and $\mathrm{d}W_{t}^{Y}$ are the same as in Example \ref{ex1}, and $J_{t}^{X}$, $J_{t}^{Y}$ are both the symmetric $\alpha$-stable processes with $\alpha_{1}=0.5$ and $\alpha_{2}=0.9$ respectively.
	\end{ex}

	
The simulations are repeated 5000 times. We compute the bias, standard deviation (SD), and mean square error (MSE) of our proposed estimators, $V(u,v), U(u,v)$ and $V'(u,v)$. Results are displayed in Table \ref{table1}.
	
Moreover, we assess the asymptotic normality of the estimators $V_n(u,v)$ and $U_n(u,v)$ via both the infeasible and the feasible limit theorems, respectively, and set $u=3.5$, $v=3.75$. In Figure \ref{figure1}, we display the histograms of the infeasible central limit theorem, and the histograms of the studentized statistic are exhibited in Figure \ref{figure2}, only for the setting in {Example \ref{ex3}}, since the figures are similar for other settings.
\begin{table}[htb!]
		\centering
		\caption{Monte Carlo Results of $V_n(u,v)$, $U_n(u,v)$ and $V'_n(u,v)$}
		
		\label{table1}
\resizebox{1.066\columnwidth}{!}{
		\begin{tabular}{cc|ccccccccc}
			\hline
			  \multicolumn{2}{c}{\multirow{2}{*}{$V_{n}(u,v)$}} & \multicolumn{3}{c}{$u=2.5$} & \multicolumn{3}{c}{$u=3.5$} & \multicolumn{3}{c}{$u=4.5$} \\
			  \multicolumn{2}{c}{} & \text{Bias} &\text{SD} &\text{MSE} &\text{Bias} & \text{SD} & \text{MSE} & \text{Bias}&\text{SD} &\text{MSE} \\
			\hline
			\multicolumn{11}{c}{\textbf{Example 1 }} \\
			
			\multirow{3}{*}{$v$} & 2.75 & 0.00024 &	0.02419	& 0.00059 & 0.00077	& 0.02390 & 0.00057 & -0.00050	& 0.02465 & 0.00061\\
			& 3.75 & 0.00046 & 0.02415	& 0.00058 & -0.00014& 0.02398 & 0.00058	& 0.00116	& 0.02447 & 0.00060\\
			& 4.75 & 0.00083 &0.02384	&0.00057  &	-0.00113& 0.02352 & 0.00055	& -0.00141	& 0.02430 & 0.00059\\
			\hline
			\multicolumn{11}{c}{\textbf{Example 2}} \\
			
			\multirow{3}{*}{$v$} & 2.75 & -0.00023&	0.02361	&0.00056 & 0.00184 & 0.02456 & 0.00061	&-0.00067 & 0.02355	& 0.00056\\
			& 3.75 & 0.00048 &	0.02298	&0.00053 & 0.00112 & 0.02271 & 0.00052	& 0.00012 & 0.02329	& 0.00054\\
			& 4.75 & -0.00012&	0.02386	&0.00057 & 0.00053 & 0.02393 & 0.00057	&-0.00126 & 0.02385	& 0.00057\\
			\hline
			\multicolumn{11}{c}{\textbf{Example 3}} \\
			
			\multirow{3}{*}{$v$} & 2.75 & -0.00036&	0.02337	& 0.00055& -0.00081	& 0.02384& 0.00057 & 0.00047 & 0.02255	& 0.00051\\
			& 3.75 & -0.00071&	0.02344	& 0.00055&	0.00082	& 0.02349& 0.00055 & -0.00056& 0.02500	& 0.00063\\
			& 4.75 & -0.00027&	0.02412	& 0.00058& -0.00011	& 0.02371& 0.00056 & -0.00048& 0.02473	& 0.00061\\
			\hline
            \hline
		 \multicolumn{2}{c}{\multirow{2}{*}{$U_{n}(u,v)$}} & \multicolumn{3}{c}{$u=2.5$} & \multicolumn{3}{c}{$u=3.5$} & \multicolumn{3}{c}{$u=4.5$} \\
		  \multicolumn{2}{c}{} & \text{Bias} &\text{SD} &\text{MSE} &\text{Bias} & \text{SD} & \text{MSE} & \text{Bias}&\text{SD} &\text{MSE} \\
			\hline
			\multicolumn{11}{c}{\textbf{Example 1}} \\
			\multirow{3}{*}{$v$} & 2.75 & 0.00071 & 0.01670	&0.00028 & 0.00040	& 0.01713 & 0.00029	& -0.00031 & 0.01667 &0.00028\\
			& 3.75 & -0.00042&	0.01655	&0.00027 & 0.00042	& 0.01655 & 0.00027	& -0.00052 & 0.01713 &0.00029\\
			& 4.75 & -0.00005&	0.01644	&0.00027 & -0.00085	& 0.01608 &	0.00026	& -0.00011 & 0.01713 &0.00029\\
			\hline
			\multicolumn{11}{c}{\textbf{Example 2}} \\
			
			\multirow{3}{*}{$v$} & 2.75 & -0.00015&	0.01670	& 0.00028 & 0.00049  & 0.01682	& 0.00028 & 0.00020	& 0.01638 & 0.00027\\
			& 3.75 & -0.00022&	0.01599	& 0.00026 &	-0.00019 & 0.01705	& 0.00029 & 0.00007	& 0.01637 & 0.00027\\
			& 4.75 & -0.00063&	0.01695	& 0.00029 &	-0.00058 & 0.01660	& 0.00028 &	0.00012	& 0.01671 & 0.00028\\
			\hline
			\multicolumn{11}{c}{\textbf{Example 3}} \\
			
			\multirow{3}{*}{$v$} & 2.75 & -0.00013&	0.01753	&0.00031 &	0.00011& 0.01647 & 0.00027 & -0.00027 &	0.01679 & 0.00028\\
			& 3.75 & 0.00032 &	0.01679	&0.00028 &	-0.0001& 0.01726 & 0.00030& 0.00064  & 0.01746	&0.00031\\
			& 4.75 & -0.00018&	0.01648	&0.00027 &	0.00048& 0.01621 & 0.00026 & -0.00067 & 0.01734	&0.00030\\
			\hline
			\hline
			 \multicolumn{2}{c}{\multirow{2}{*}{$V'_{n}(u,v)$}} & \multicolumn{3}{c}{$u=2.5$} & \multicolumn{3}{c}{$u=3.5$} & \multicolumn{3}{c}{$u=4.5$} \\
			 \multicolumn{2}{c}{} & \text{Bias} &\text{SD} &\text{MSE} &\text{Bias} & \text{SD} & \text{MSE} & \text{Bias}&\text{SD} &\text{MSE} \\
			\hline
			\multicolumn{11}{c}{\textbf{Example 1 }} \\
			
			\multirow{3}{*}{$v$} & 2.75 & -0.00038 & 0.01670 & 0.00028&	-0.00028 &	0.01717& 0.00029& -0.00076& 0.01697& 0.00029\\
			& 3.75 & 0.00057 &	0.01654	&0.00027&	0.00064& 0.01671& 0.00028 & 0.00054 & 0.01644 & 0.00027\\
			& 4.75 & -0.00008&	0.01689	&0.00029&	-0.00005&	0.01673& 0.00028& -0.00081&0.01654&0.00027\\
			\hline
			\multicolumn{11}{c}{\textbf{Example 2}} \\
			
			\multirow{3}{*}{$v$} & 2.75 &-0.00020&0.01660&0.00028&-0.00023&0.01679&0.00028&0.00022&0.01696&0.00029\\

			& 3.75& 0.00022&	0.01643&0.00027&0.00029&0.01697&0.00029&0.00015&0.01676& 0.00028\\
			& 4.75& -0.00071&	0.01683&	0.00028&	-0.00006&	0.01708&0.00029&0.00062&0.01660&0.00028\\
			\hline
			\multicolumn{11}{c}{\textbf{Example 3}} \\
			
			\multirow{3}{*}{$v$} & 2.75 & -0.00053&	0.01719&0.00030&-0.00031&0.01709&0.00029&0.00024& 0.01626&0.00026\\
			& 3.75 &-0.00049&	0.01684&0.00028&-0.00043&	0.01653&	0.00027&	-0.00032&	0.01647&	0.00027\\
			& 4.75 & 0.00015&	0.01678&0.00028&0.00030&	0.01691&0.00029&-0.00048&0.01669&0.00028\\
			\hline
		\end{tabular}
}
\end{table}

	\begin{figure}[!htbp]
		\centering
	
		\subfigure{
			\includegraphics[width=7.5cm,height=8cm]{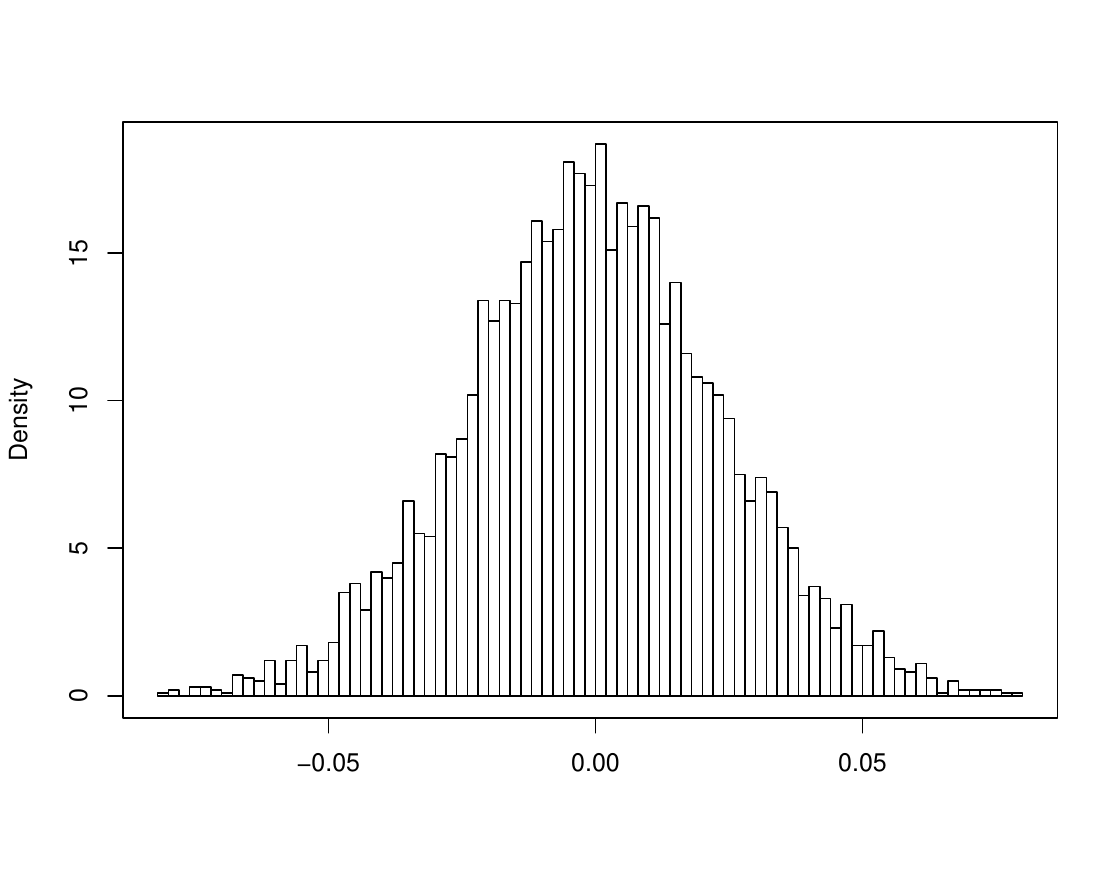}
		}
		\subfigure{
			\includegraphics[width=7.5cm,height=8cm]{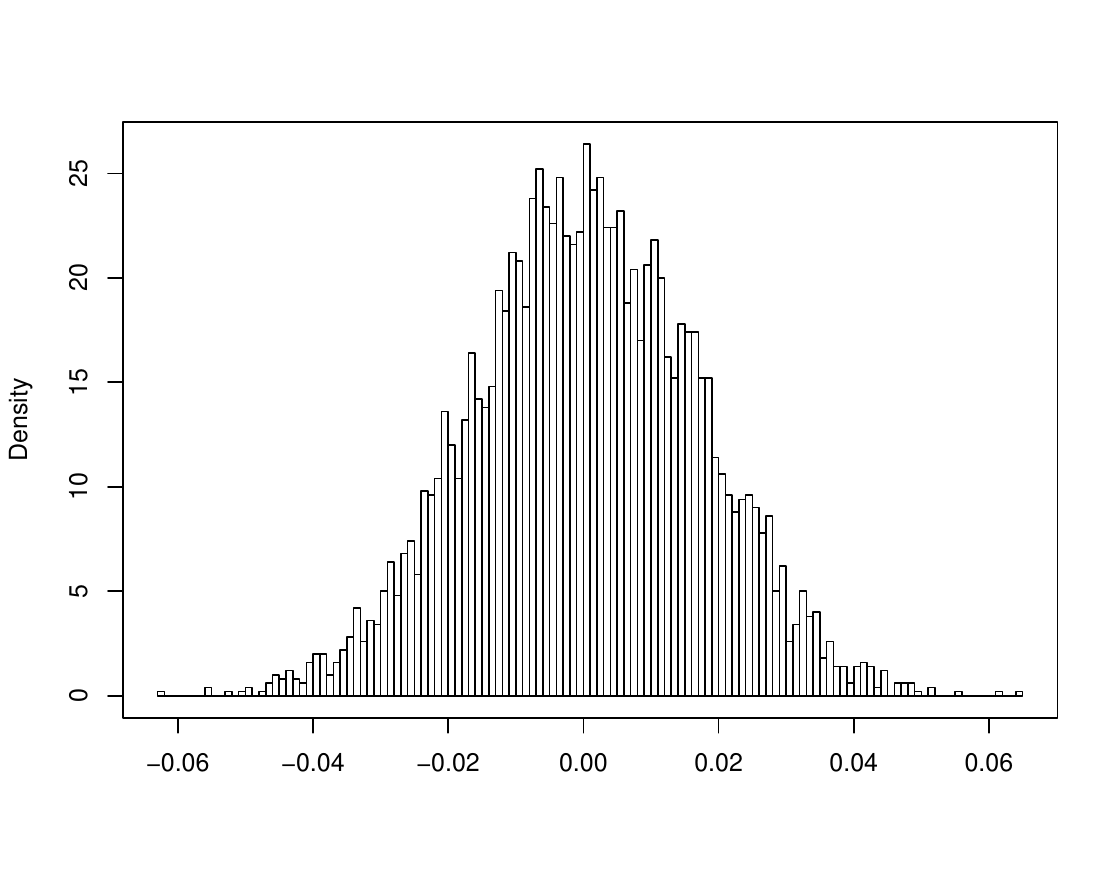}
		}
		\caption{Left panel: The histograms of $V_n(u,v)-\int_0^1\mbox{e}^{-\langle (u, v), ((\sigma_{s}^{X})^2, (\sigma_{s}^{Y})^2)\rangle}\mathrm{d}s$; Right panel: The histograms of $U_n(u,v)-\int_0^1\mbox{e}^{-\langle (u, v), ((\sigma_{s}^{X})^2, (\sigma_{s}^{Y})^2)\rangle}\mathrm{d}s$.}
		\label{figure1}
	\end{figure}
	
	\begin{figure}[!htbp]
		\centering
		
		\subfigure{
			\includegraphics[width=7.5cm,height=8cm]{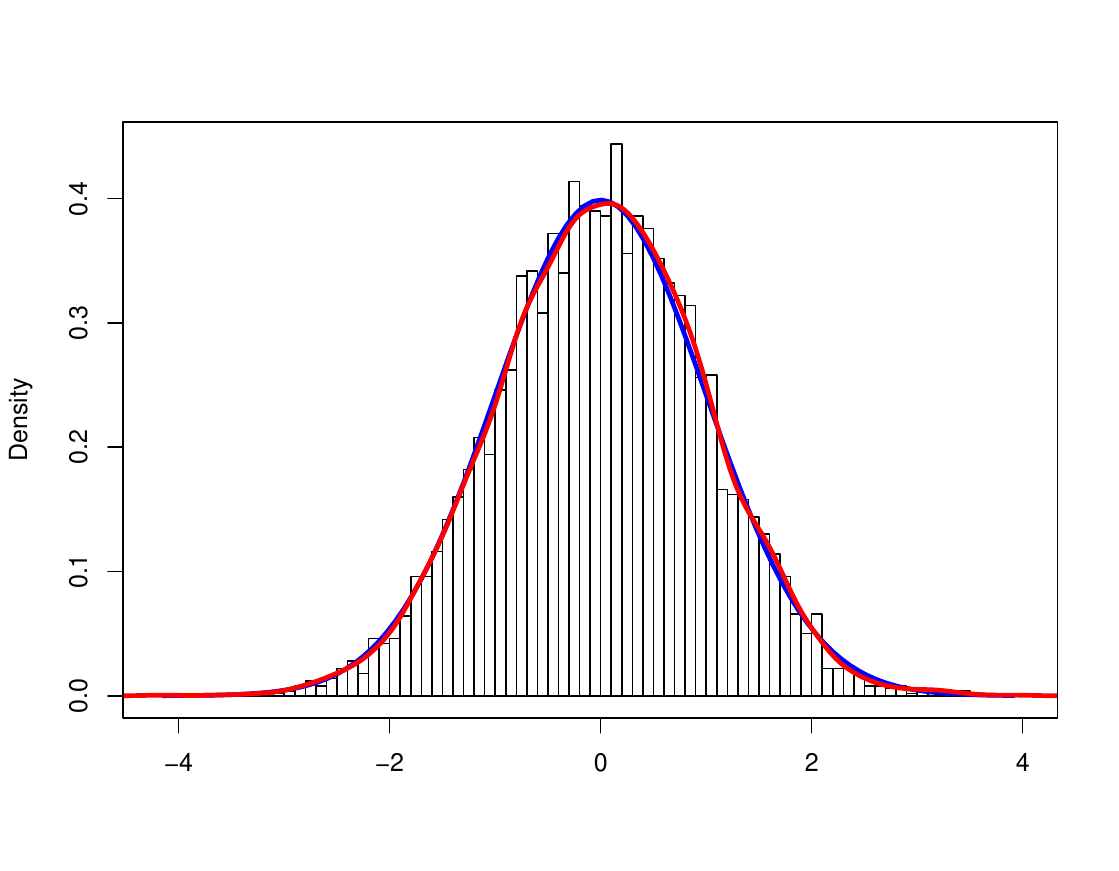}
		}
		\subfigure{
			\includegraphics[width=7.5cm,height=8cm]{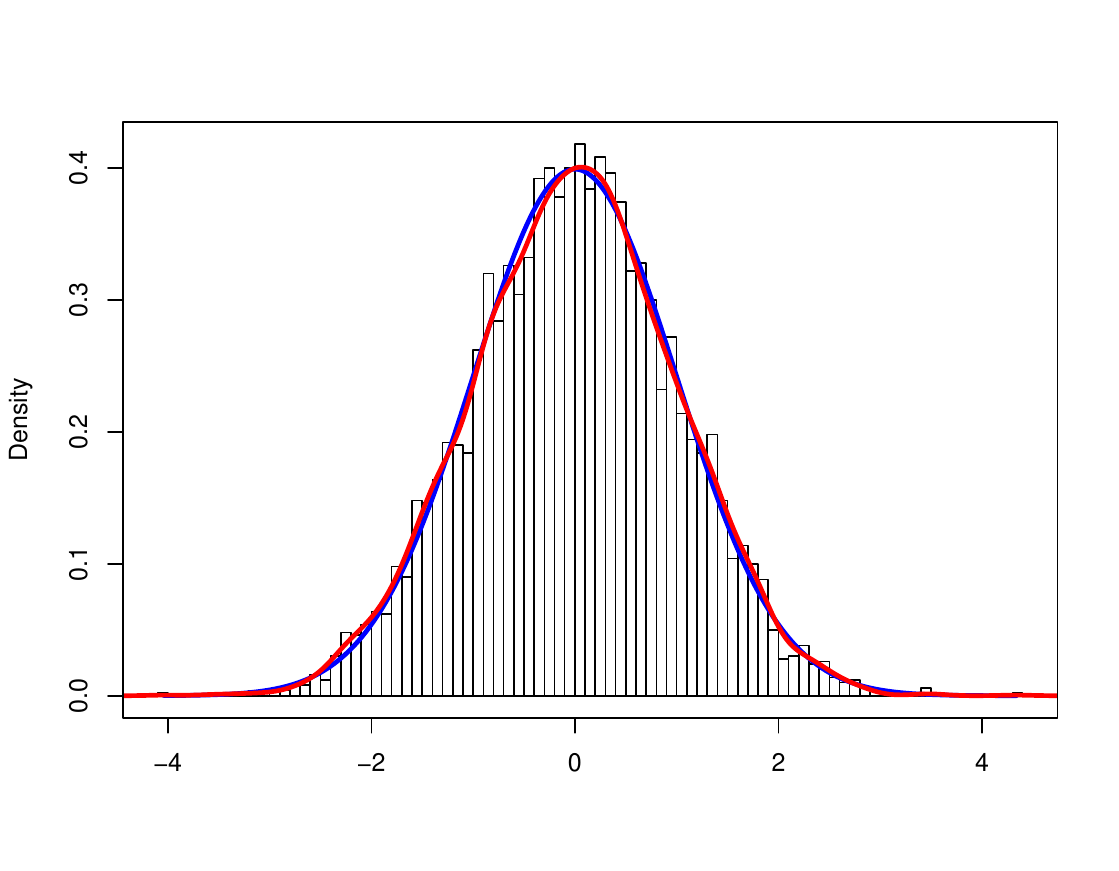}
		}
		\caption{Left panel: The histograms of $\frac{1}{\sqrt{\Delta_n\tilde{\Gamma}_{n}}}\left(V_n(u,v)-\int_0^1\mbox{e}^{-\langle (u, v), ((\sigma_{s}^{X})^2, (\sigma_{s}^{Y})^2)\rangle}\mathrm{d}s\right)$; Right panel: The histograms of $\frac{1}{\sqrt{\Delta_n\tilde{\Gamma}_{n}}}\left(U_n(u,v)-\int_0^1\mbox{e}^{-\langle (u, v), ((\sigma_{s}^{X})^2, (\sigma_{s}^{Y})^2)\rangle}\mathrm{d}s\right)$.}
		\label{figure2}
	\end{figure}
	
We have the following observations from this simulation.
	\begin{enumerate}
		\item The bias of estimators $V_n(u,v)$, $U_n(u,v)$, $V'_n(u,v)$ are very close to zero, which demonstrates all of our estimators are consistent.
		\item From Table \ref{table1}, $U_n(u,v)$ always performs better than the others. This is in line with our theory.
        \item We see that the results for $U_{n}(u,v)$ and $V'_{n}(u,v)$ are very similar from Table \ref{table1}, which is consistent with the limiting theory we obtained in Section \ref{asymptotic}.
		\item All estimators are robust to the Poisson jumps and $\alpha$-stable jumps with $\alpha<1$.
		\item Figures \ref{figure1} and \ref{figure2} show that the estimators behave like a normal distribution, justifying our asymptotic theory.
	\end{enumerate}

\subsection{Non-synchronous observation}
In real financial markets, we usually only observe asynchronous high-frequency data. The issue has gained significant attention in the literature of financial econometrics, particularly with the advent of large-scale high-dimensional data in recent years. Here, we provide a way to adjust our estimator to handle this issue. As we can see in Figure \ref{figure8}, $t_{i}^{X}$, $t_{i}^{Y}$ are the $i$-th and $t_{n}^{X}$, $t_{n^{*}}^{Y}$ are the last observation time in $[0, T]$, of the price processes $X$ and $Y$, respectively.
\begin{figure}[!htbp]
\flushright
\includegraphics[width=16cm,height=4cm]{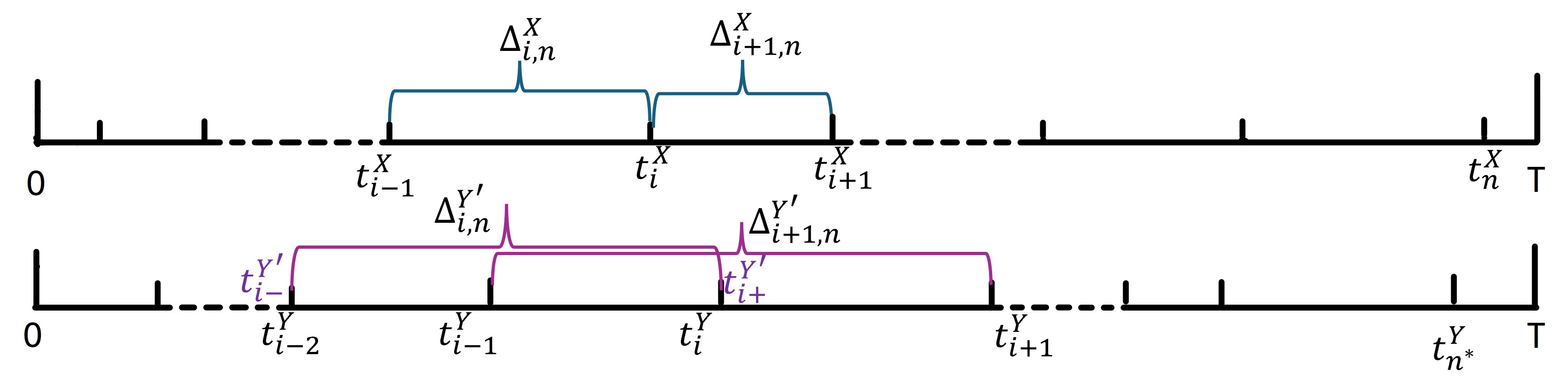}
\caption{The Non-synchronous observations of price processes $X$ and $Y$.}
\label{figure8}
\end{figure}

Here, we adjust the estimator $U_{n}$ to $U'_{n}$ to accommodate this type of data. Define
\begin{equation*}
U'_{n}=\sum_{i=1}^{N_T^X}\Delta_{i,n}^{X}\cos\left(\sqrt{2u}\frac{\Delta_i^n{X}}{\sqrt{\Delta_{i,n}^{X}}}+\sqrt{2v}\frac{\Delta_{i}^{n}{Y'}}{\sqrt{\Delta_{i,n}^{Y'}}}\right),
\end{equation*}
where $T=1$, $N_T^X$ is the Poisson process with $\mathbb{E}[N_{T}^{X}]=1760$, and we let $\Delta_{i,n}^{X}=t^X_{i}-t^X_{i-1}$, $\Delta_i^nX=X_{t^X_{i}}-X_{t^X_{i-1}}$, and $\Delta_i^nY'=Y_{t^Y_{i+}}-Y_{t^Y_{i-}}$, with
$$t_{i-}^{Y}=\max_j\{t_{j}^{Y}: t_{j}^{Y}\leq t_{i-1}^{X}\},~t_{i+}^{Y}=\min_j\{t_{j}^{Y}: t_{j}^{Y}\geq t_{i}^{X}\}, \Delta_{i,n}^{Y'}=t_{i+}^{Y'}-t_{i-}^{Y'}.$$
Mathematically, $\Delta_{i,n}^{Y'}$ is the smallest cover of $\Delta_{i,n}^{X}$ in process $Y$ and  $\Delta_{i}^{n}{Y'}$ is the corresponding increment. The volatility and price processes settings are the same as Example \ref{ex1}, and the results of two other Example settings are similar. The performance of the adjusted estimator under $u=3.5$, $v=3.75$ are displayed in Table \ref{table6} and Figure \ref{figure3}, respectively.
	\begin{table}[htb!]
		\centering
		\caption{Monte Carlo Results of $U_n^{'}(u,v)$}
		\label{table6}
\resizebox{1.0\columnwidth}{!}{
		\begin{tabular}{cc|ccccccccc}
			\hline
			 & & \multicolumn{3}{c}{$u=2.5$} & \multicolumn{3}{c}{$u=3.5$} & \multicolumn{3}{c}{$u=4.5$} \\
			 & & \text{Bias} &\text{SD} &\text{MSE} &\text{Bias} & \text{SD} & \text{MSE} & \text{Bias}&\text{SD} &\text{MSE} \\
			\hline
			\multicolumn{11}{c}{\textbf{Example 1 }} \\
			\multirow{3}{*}{$v$} & 2.75 & 0.00086 &	0.02396 & 0.00057 &-0.00059	& 0.02363 & 0.00056 & -0.00092	& 0.02368 & 0.00056\\
			& 3.75 & -0.00015 & 0.02383	& 0.00057 & -0.00081& 0.02429 & 0.00059	& -0.00085	& 0.02336 & 0.00055\\
			& 4.75 & -0.00012 &0.02322	&0.00055  &	-0.00030& 0.02365 & 0.00056	& 0.00010	& 0.02396 & 0.00057\\
			\hline
	\end{tabular}
}
	\end{table}
\begin{figure}[!htbp]
		\centering
\includegraphics[width=7.5cm,height=8cm]{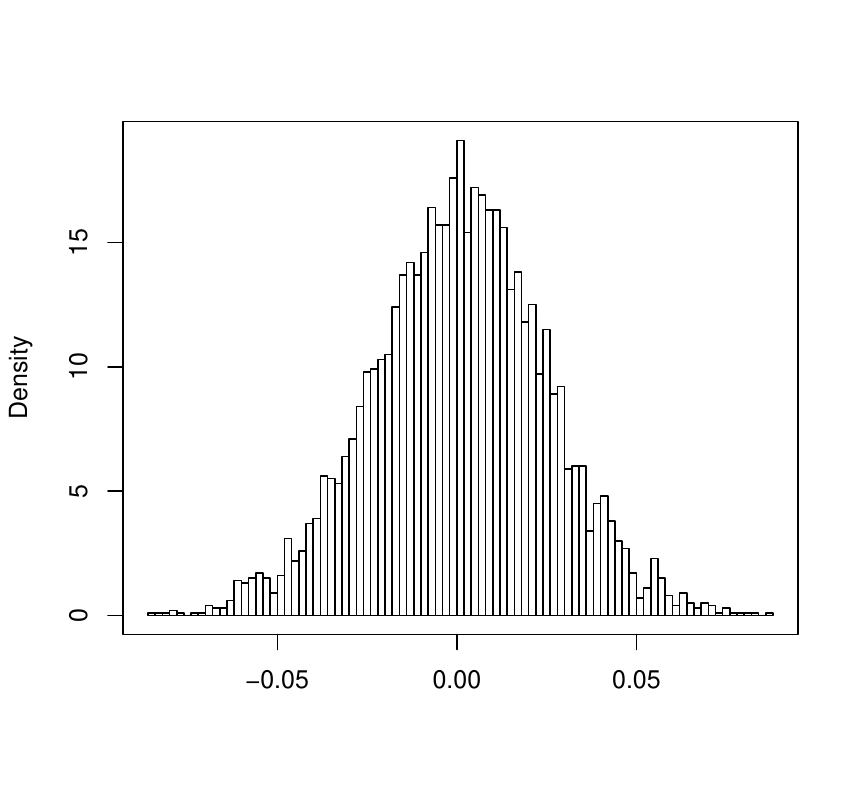}
\caption{The histograms of $U_{n}'(u,v)-\int_0^1\mbox{e}^{-\langle (u, v), ((\sigma_{s}^{X})^2, (\sigma_{s}^{Y})^2)\rangle}\mathrm{d}s$.}
\label{figure3}
\end{figure}

Table \ref{table6} presents the bias of adjusted estimator $U_{n}^{'}$. We see that they are also very close to zero but mostly larger than $U_{n}$, which demonstrates estimator $U_{n}^{'}$ can be used for non-synchronous data in practice, as MSE and standard deviation perform well. In addition, the histogram of $U_{n}'(u,v)-\int_0^1\mbox{e}^{-\langle (u, v), ((\sigma_{s}^{X})^2, (\sigma_{s}^{Y})^2)\rangle}\mathrm{d}s$ is very close to the normal distribution, as seen from Figure \ref{figure3}.

\subsection{Test for the dependence of volatility processes}\label{application}
Based on the above results, we test the following hypothesis:
$$
H_{0}:\sigma_{s}^{x}\ and\ \sigma_{s}^{y}\ are\ independent \quad v.s. \quad H_{1}:\sigma_{s}^{x}\ and\ \sigma_{s}^{y}\ are\ dependent.
$$

We use the test statistic $||\hat{\cal S}_{T,n}||^2$ in Corollary \ref{cor1}.
To approximate the limiting distribution $||{\boldsymbol{\gamma}} \cdot  {\boldsymbol \Phi}(u,v)||^2$, we first partition the rectangle $[0,1]\times[0,1]$ into 100 same sub-rectangle, and let $u_{i}=i/10,v_{j}=j/10$, $i,j=1,2...10$. Then, we approximate the limit distribution of the test statistic by
 \begin{equation}\label{eq6}
\int_{0}^{1}\int_{0}^{1}\left[{\boldsymbol{\gamma}}(u,v) \cdot  {\boldsymbol \Phi}(u,v)\right]^{2}{\rm{d}}u{\rm{d}}v\approx \frac{1}{100}\sum_{i=1}^{10}\sum_{j=1}^{10}\left[{\boldsymbol{\gamma}}(u_{i},v_{j}) \cdot  {\boldsymbol \Phi}(u_{i},v_{j})\right]^{2}\stackrel{d}{=}\frac{1}{100}\sum_{i=1}^{100}\hat{\pi}_{i}\chi_{i}^{2}.
 \end{equation}
 where $({\boldsymbol{\gamma}}(u_{1},v_{1}) \cdot  {\boldsymbol \Phi}(u_{1},v_{1}),\cdots,{\boldsymbol{\gamma}}(u_{10},v_{10}) \cdot  {\boldsymbol \Phi}(u_{10},v_{10}))$ is, conditionally on $\cal{F}$, a one-hundred dimensional normal distribution with conditional covariance matrix $\hat{C}((u_{i},v_{j}),(u_{i'},v_{j'}))$ where $1\leq i,j,i',j'\leq 10$ and $\chi_{i}^{2}$ are independent and $\chi_{1}^{2}$ distributed, which are defined on an extension of the original probability space and independent from  $\cal{F}$, the $\hat{\pi}_{i}'s$ are the eigenvalues of the matrix $\hat{C}((u_{i},v_{j}),(u_{i'},v_{j'}))$. However, the eigenvalues of the covariance matrix may have a small number of negative values in practice. If this happens, we will discard the negative eigenvalues and only keep the positive in the above Equation (\ref{eq6}). And also the test statistic $||\hat{\cal S}_{T,n}||^2$ can be approximately by Riemann sum similarly on the rectangle $[0,1]\times[0,1]$. The test procedure is demonstrated with the following simulated processes.
\begin{ex}\label{ex4}
We use the data-generating processes of Example \ref{ex1} with the stationary volatility processes. Now, we generate $\tau_{t}^{x}$ and $\tau_{t}^{y}$ from a discrete auto-regressive processes. That is, we consider
 $$\tau_{t}^{x}=0.5\tau_{t-1}^{x}+\epsilon_{t}^{x},~~\tau_{t}^{y}=0.7\tau_{t-1}^{y}+\epsilon_{t}^{y},~{\text with}~ \epsilon_{t}^{y}=\rho_{t}'\epsilon_{t}^{x}+\sqrt{1-\rho_{t}'^{2}}\epsilon_{t}^{\star},$$
where, $\epsilon_{t}^{x}$ and $\epsilon_{t}^{\star}$ are two mutually independent standard normal random variables. It follows that $\epsilon_{t}^{y}$ is independent of $\epsilon_{t}^{x}$ if $\rho_{t}'=0$, otherwise they are dependent.
\end{ex}
Under the null hypothesis $H_{0}$, we let $u=0.5$, $v=0.5$, and consider the following four scenarios of the $T$ and $\Delta_n$, the size and power of the test are exhibited in Table \ref{table5},
\begin{itemize}
\item $T=22$(one month), $\Delta_{n}=1/390$,
\item $T=44$(two months), $\Delta_{n}=1/390$,
\item $T=44$(two months), $\Delta_{n}=1/780$,
\item $T=66$(one quarter), $\Delta_{n}=1/780$.
\end{itemize}

From Table \ref{table5}, we see that when the number of days $T$ increases and the increment $\Delta_{n}$ decreases, the power of the test becomes closer to 1, which shows that the performance of the test is satisfactory and the simulation studies are line with our asymptotic results.
\begin{table}[htb!]
		\centering
{
		\caption{The performance of test statistic}\label{table5}
		\vspace{0.2cm}
\begin{tabular}{r|cccc|cccc}
\toprule  
 &\multicolumn{4}{c}{$T, n, \alpha=0.05$}& \multicolumn{4}{c}{$T, n, \alpha=0.10$}\\
$\rho'$& $22, 390$ &44, 390 & 44, 780 &66, 780& $22, 390$ &44, 390 &44, 780 &66, 780\\
\midrule  
0   &0.06   &0.04     & 0.04   & 0.06     &0.10  &0.12 &0.11  & 0.08  \\
0.2 & 0.26    & 0.22   &0.45   &0.72  &0.32  &0.36  &0.66  &0.80 \\
0.5 & 0.47    & 0.52   &0.91   &0.96  &0.63  &0.74  &0.96  &0.99 \\
-0.5 & 0.51    & 0.59   &0.89  &0.99 &0.63  &0.74  &0.94  &0.99 \\
0.8 & 0.65    & 0.82   &0.99   &1.00 &0.75  &0.92  &1.00  &1.00 \\
\bottomrule 
\end{tabular}
}
\end{table}

\section{Empirical study}\label{empirical}
In this section, we implement our proposed test by using the high-frequency data consisting of several stocks of the S\&P 500 index (INX). The period for the data used is from October 1-31, 2019, and the sampling frequency is one minute. We selected five stocks from each of the following four sectors:
\begin{itemize}
\item Information technology,
\item Finance,
\item Consumer staples,
\item Real estate.
\end{itemize}
The symbol of chosen stocks is shown in the Table \ref{table10}. For the convenience of calculation, we rescale the return by multiplying by fifteen so that the resulting spot volatility is not affected by the approximation of the cosine function while the correlation between volatilities remains unchanged.  We utilize the test statistic introduced in Section \ref{longspan} and compute the $p$-values. The results are presented in Table \ref{table10}.

From Table \ref{table10}, we see at the nominal level $\alpha$ as 0.05, the null hypotheses of about 75\% pairs have been rejected. Namely, about 3 quarters of the stocks are dependent on each other among these 20 stocks. Furthermore, the dependence patterns within sectors and between different sectors are also different. For example, the volatilities of stocks in the sector Consumer Staples are more dependent on other sectors, showing the smallest averaged $p$-values. This phenomenon is probably due to the high market capitalizations of these companies and their significant connections to the daily lives of individuals. The volatilities of stocks in the Real Estate sector are relatively independent since, on average, the $p$-values are the largest. Since the lower stock market capitalization of this sector, especially SPG, its volatility is almost independent of all other stocks. Finally, the volatilities of all the giant stocks (APPL, MSFT, JPM, KO, WMT) are highly dependent on others regardless of whether they are in the same sector or different. These findings are generally aligned with our intuition.
\begin{table}[htb!]
		\centering
		\caption{$p$-value of the independence test for volatilities}
		\label{table10}
\resizebox{1.0\columnwidth}{!}{
		\begin{tabular}{c|ccccc|ccccc|ccccc|ccccc}
			\toprule[1.5pt]
Sector & \multicolumn{5}{c}{Information Technology} & \multicolumn{5}{c}{Finance}&\multicolumn{5}{c}{Consumer Staples} &\multicolumn{5}{c}{Real Estate}\\
\midrule[1.2pt]
Symbol & \text{APPL} &\text{MSFT} &\text{ORCL} &\text{IBM} & \text{ADBE} & \text{BAC} & \text{JPM}&\text{MS} &\text{WFC} &\text{PNC} &\text{KO} &\text{WMT} &\text{ADM} &\text{COST} &\text{KMB} &\text{O} &\text{SPG} &\text{KIM} &\text{CCL} &\text{WY}\\
			\midrule[1.2pt]
\text{APPL}&&0.02&0.03&0.01&$<$0.01&$<$0.01&$<$0.01&0.02&$<$0.01&$<$0.01&$<$0.01&$<$0.01&$<$0.01&0.01&0.05&0.08&0.14&0.02&0.03& 0.04\\
\text{MSFT}&&&0.02 &0.03& 0.11& 0.02  & $<$0.01  & 0.01  & 0.03  & 0.04  &$<$0.01 & $<$0.01 &$<$0.01&0.01& 0.30 & $<$0.01& 0.11& 0.01&$<$0.01 & 0.03\\
\text{ORCL}&&&&	0.01 & 0.03  & 0.03  & 0.02  & 0.03  & $<$0.01  & 0.15  &$<$0.01& $<$0.01 &$<$0.01	&0.01 &$<$0.01 & 0.03  & 0.25  & 0.17 &0.03 &0.04 \\
\text{IBM}&&&&& 0.01 &0.18&$<$0.01&0.10&0.24&0.32&$<$0.01&$<$0.01&$<$0.01&$<$0.01&0.04&$<$0.01&0.16 &0.01 & 0.04  &0.03  \\
\text{ADBE}&&&&&&0.17& 0.01& 0.04 &0.32& 0.13&0.03&0.03&$<$0.01&0.04&0.22&0.11&0.43&0.08 &0.03&$<$0.01 \\
\hline
\text{BAC}&&&&&&&$<$0.01&0.08&0.18&0.04&$<$0.01&0.01&$<$0.01&$<$0.01&$<$0.01&0.01&0.11&$<$0.01&0.01&0.01\\
\text{JPM}&&&&&&&&0.04&0.21&$<$0.01&$<$0.01&$<$0.01&0.01&$<$0.01&$<$0.01&$<$0.01&0.02&$<$0.01 &$<$0.01&0.02 \\
\text{MS} &&&&&&&&&0.41&0.52&$<$0.01&$<$0.01&$<$0.01&0.03&0.14&0.01&0.20&$<$0.01&0.30&$<$0.01\\
\text{WFC}&&&&&&&&&&0.02&$<$0.01&$<$0.01   &$<$0.01	 & $<$0.01  &0.01 &$<$0.01&0.15&$<$0.01&0.25&$<$0.01 \\
\text{PNC}&&&&&&&&&&&$<$0.01&$<$0.01 &$<$0.01&0.02&0.39&$<$0.01&0.11&$<$0.01&0.34&0.02 \\
\hline
\text{KO} &&&&&&&&&&&&$<$0.01 &$<$0.01&$<$0.01 &$<$0.01&$<$0.01&0.15&0.01&0.01&0.01 \\
\text{WMT}&&&&&&&&&&&&&$<$0.01&0.03&0.45&$<$0.01&0.04&$<$0.01 &0.02&0.01\\
\text{ADM}&&&&&&&&&&&&&&0.02 & $<$0.01&$<$0.01&0.03&0.01&$<$0.01 &0.01 \\
\text{COST}&&&&&&&&&&&&&&&0.55&0.33&0.70&0.18&0.04&0.45 \\
\text{KMB} &&&&&&&&&&&&&&&&$<$0.01& 0.60&$<$0.01 &0.09&$<$0.01\\
\hline
\text{O}   &&&&&&&&&&&&&&&&&0.47&0.03&0.03&$<$0.01\\
\text{SPG} &&&&&&&&&&&&&&&&&&0.04&0.78&0.13 \\
\text{KIM} & & & & & & & & & & & & & & & & & & &0.16 &0.20\\
\text{CCL} & & & & & & & & & & & & & & & & & & & &0.55 \\
\text{WY}  & & & & & & & & & & & & & & & & & & & &   \\
\bottomrule[1.5pt]
	\end{tabular}
}
	\end{table}

\section{Conclusion}\label{conclusion}
In this paper, we study the limiting behavior of the realized joint Laplace transform of volatilities by using non-overlapped and overlapped high-frequency data, including consistency and asymptotic normality. Moreover, we derive a central limit theorem for the joint Laplace transform under the assumption of the stationary volatility processes. Based on the results, we propose a consistent test for the dependence between two volatility processes. Extensive simulation studies verify the finite sample performance of the proposed theory. We implement the test procedure using a real high-frequency data set.

This paper highlights some future research. For example, we may
develop large sample theory for asynchronous high-frequency data, study the related limiting behavior and independence tests for the multivariate volatilities, and consider a more general dependence structure of price processes, such as nonlinear dependence, etc.

\section{Appendix: proofs}\label{appendix} This section contains all the proofs. Throughout the proofs, we denote by $C$ a generic constant that may change from line to line. We sometimes emphasize its dependence on some parameter $p$ by writing $C_p$. As is typical in this kind of problem, by a standard localization procedure as proved in \cite{jacod2012discretization}.
	
We denote
	\begin{equation*}
		\alpha_i^Z:=\alpha_{t_i^n}^Z,~~\text{for}~\alpha=b, \tilde{b}, v, v', ~Z=X, Y,~~{\cal G}_i={\cal F}_{t_{2i}^n}, ~i\geq 0.
	\end{equation*}
	Let
	\begin{align*}
		\xi_{i}^{(1)}(u,v)=&2\Delta_n^{\frac12}\left\{\cos\left(\frac{\sqrt{2u}\Delta_i^n{X}+\sqrt{2v}\Delta_{i+1}^n{Y}}{\sqrt{\Delta_n}}\right)- \cos\left(\frac{\sqrt{2u}\sigma_{i-1}^X\Delta_i^n{W^X}+\sqrt{2v}\sigma_{i-1}^Y\Delta_{i+1}^n{W^Y}}{\sqrt{\Delta_n}}\right)\right\},\\
		\xi_{i}^{(2)}(u,v)=&2\Delta_n^{\frac12}\left\{\cos\left(\frac{\sqrt{2u}\sigma_{i-1}^X\Delta_i^n{W^X}+\sqrt{2v}\sigma_{i-1}^Y\Delta_{i+1}^n{W^Y}}{\sqrt{\Delta_n}}\right)-\mbox{e}^{-\langle (u, v), ((\sigma_{i-1}^X)^2, (\sigma_{i-1}^Y)^2)\rangle}\right\},\\
		\xi_{i}^{(3)}(u,v)=&\Delta_n^{-\frac12}\left\{2\Delta_n\mbox{e}^{-\langle (u, v), ((\sigma_{i-1}^X)^2, (\sigma_{i-1}^Y)^2)\rangle}-\int_{t_{i-1}^n}^{t_{i+1}^n}\mbox{e}^{-\langle (u, v),((\sigma_s^X)^2, (\sigma_s^Y)^2)\rangle}\mathrm{d}s\right\}.
	\end{align*}
	For every $t$, $u$ and $v$, we can rewrite the left side of \eqref{equ3} as
	\begin{equation*}
		\frac{1}{\sqrt{\Delta_n}}\left(V_n(u,v)-\int_0^T\mbox{e}^{-\langle (u, v), ((\sigma_{s}^{X})^2, (\sigma_{s}^{Y})^2)\rangle}\mathrm{d}s\right)=\sum_{i=1}^{\lfloor n/2\rfloor}\sum_{j=1}^{3}\xi_{2i-1}^{(j)}(u,v).
	\end{equation*}
	
	Before giving the proofs of the main results, we first present the following lemma for the approximation error.
	\begin{lem}\label{lemma1}
		Let Assumptions \ref{asu1}-\ref{asu2} hold, we have
		\begin{equation*}
			\sum_{i=1}^{\lfloor n/2\rfloor}\xi_{2i-1}^{(1)}(u,v)\stackrel{P}{\longrightarrow}0, ~~
			\sum_{i=1}^{\lfloor n/2\rfloor}\xi_{2i-1}^{(3)}(u,v)\stackrel{P}{\longrightarrow}0.\label{xi1}
		\end{equation*}
	\end{lem}
	\begin{pf}
	First for $\sum_{i=1}^{\lfloor n/2\rfloor}\xi_{2i-1}^{(3)}(u,v)$, note that $\xi_{2i-1}^{(3)}(u,v)$ can be written as $\sum_{j=1}^{3}\xi_{2i-1}^{(3,j)}(u,v)$, with
	\begin{align*}
		\xi_{2i-1}^{(3,1)}(u,v)=&\Delta_n^{-\frac12}\int_{t_{2i-2}^n}^{t_{2i}^n}\left\langle k(\sigma_{\ast}^{X},\sigma_{\ast}^{Y},u,v)-k(\sigma_{2i-2}^{X},\sigma_{2i-2}^{Y},u,v),
		(\sigma_{2i-2}^{X}-\hat{\sigma}_{s}^{X},\sigma_{2i-2}^{Y}-\hat{\sigma}_{s}^{Y})\right\rangle\mathrm{d}s,\\
		\xi_{2i-1}^{(3,2)}(u,v)=&\Delta_n^{-\frac12}\int_{t_{2i-2}^n}^{t_{2i}^n}\left\langle k(\sigma_{2i-2}^{X},\sigma_{2i-2}^{Y},u,v),  (\sigma_{2i-2}^{X}-\hat{\sigma}_{s}^{X},\sigma_{2i-2}^{Y}-\hat{\sigma}_{s}^{Y})\right\rangle\mathrm{d}s,\\
		\xi_{2i-1}^{(3,3)}(u,v)=&\Delta_n^{-\frac12}\int_{t_{2i-2}^n}^{t_{2i}^n}\mbox{e}^{\langle (u,v), ((\hat{\sigma}_{s}^{X})^{2},(\hat{\sigma}_{s}^{Y})^{2})\rangle}-\mbox{e}^{\langle (u,v), ((\sigma_{s}^{X})^{2},(\sigma_{s}^{Y})^{2})\rangle}\mathrm{d}s,
	\end{align*}
	where $k(x,y,u,v)=(-2ux\mbox{e}^{-ux^2-vy^2},-2vy\mbox{e}^{-ux^2-vy^2})$, $\sigma_{\ast}^{Z}$ is a number between $\sigma_{2i-2}^{Z}$ and $\hat{\sigma}_{s}^{Z}$, $ s\in[t_{2i-2}^{n},t_{2i}^{n}]$, $Z=X,Y$, and
	\begin{equation*}
		\hat{\sigma}_{s}^{Z}=\sigma_{2i-2}^{Z}+\int_{t_{2i-2}^n}^{s}v_{r}^{Z}\mathrm{d}W_{r}^{Z}+\int_{t_{2i-2}^n}^{s}v_{r}^{'Z}\mathrm{d}W_{r}^{'Z}+\int_{t_{2i-2}^n}^{s}\int_{\mathbb{R}}\delta^{'Z}(r-,z)\tilde{\mu}^{'Z}(\mathrm{d}r,\mathrm{d}z).
	\end{equation*}
	We denote the Euclidean space $\mathbb{R}^2$ with the standard norm $\|(x_1, x_2)\|:=\sqrt{x_1^2+x_2^2}$, note that $\|k(x,y,u,v)\|\leq C_{u,v}$ and $\|k(x_{1},y_{1},u,v)-k(x_{2},y_{2},u,v)\|\leq C_{u,v}\sqrt{(x_{1}-x_{2})^{2}+(y_{1}-y_{2})^{2}}$, for $x,y,x_{1},x_{2},y_{1},y_{2}\in\mathbb{R}$ and $u,v\geq 0$.
	For $\xi_{2i-1}^{(3,1)}(u,v)$, we have
	\begin{align}\label{xi31part}
		\mathbb{E}\left[\big|\xi_{2i-1}^{(3,1)}(u,v)\big|\right]
		\leq&\Delta_n^{-\frac12}\mathbb{E}\left[\sup\limits_{s}\left\|k(\sigma_{\ast}^{X},\sigma_{\ast}^{Y},u,v)-k(\sigma_{2i-2}^{X},\sigma_{2i-2}^{Y},u,v)\right\|\right.\nonumber\\\nonumber
		&\left.\times\int_{t_{2i-2}^n}^{t_{2i}^n}\left\|(\sigma_{2i-2}^{X}-\hat{\sigma}_{s}^{X},\sigma_{2i-2}^{Y}-\hat{\sigma}_{s}^{Y})\right\|\mathrm{d}s\right]\\\nonumber
		\leq&C_{u,v}\Delta_n^{-\frac12}\sqrt{\mathbb{E}\left[\sup\limits_{s}\left((\hat{\sigma}_s^{X}-\sigma_{2i-2}^{X})^2+(\hat{\sigma}_s^{Y}-\sigma_{2i-2}^{Y})^2\right)\right]}\\
		&\times\sqrt{2\Delta_n\mathbb{E}\left[\int_{t_{2i-2}^n}^{t_{2i}^n}\left\|(\sigma_{2i-2}^{X}-\hat{\sigma}_{s}^{X},\sigma_{2i-2}^{Y}-\hat{\sigma}_{s}^{Y})\right\|^2\mathrm{d}s\right]}.
	\end{align}
	By the Burkholder-Davis-Gundy inequality and Assumption \ref{asu2}, we have
	\begin{align}\label{supsigmax}
		&\mathbb{E}\left[\sup\limits_{s}(\hat{\sigma}_s^{X}-\sigma_{2i-2}^{X})^2\right]\nonumber\\
		\leq&C\mathbb{E}\left[\int_{t_{2i-2}^n}^{t_{2i}^n}(v_{r}^{X})^2\mathrm{d}r+\int_{t_{2i-2}^n}^{t_{2i}^n}(v_{r}^{'X})\mathrm{d}r+\int_{t_{2i-2}^n}^{t_{2i}^n}\int_{\mathbb{R}}\delta^{'X}(r-,x)\nu(\mathrm{d}x)\mathrm{d}r\right]\nonumber\\
		\leq&C\Delta_n,
	\end{align}
	and
	\begin{equation}\label{supsigmay}
		\mathbb{E}\left[\sup\limits_{s}(\hat{\sigma}_s^{Y}-\sigma_{2i-2}^{Y})^2\right]\leq C\Delta_n.
	\end{equation}
	Therefore, it follows from \eqref{xi31part}, \eqref{supsigmax} and \eqref{supsigmay} that
	\begin{equation}\label{estxi31}
		\mathbb{E}\left[\big|\xi_{2i-1}^{(3,1)}(u,v)\big|\right]\leq C_{u,v}\Delta_n^{3/2}.
	\end{equation}
	It is easy to check that
	\begin{equation}\label{estxi32}
		\mathbb{E}\left[\xi_{2i-1}^{(3,2)}(u,v)\big|\mathcal{G}_{i-1}\right]=0,
	\end{equation}	
	and
	\begin{align}\label{est2xi32}
		\mathbb{E}\left[\big|\xi_{2i-1}^{(3,2)}(u,v)\big|^{2}\right]
		\leq&C_{u,v}\mathbb{E}\left[\int_{t_{2i-2}^n}^{t_{2i}^n}\left\|(\sigma_{2i-2}^{X}-\hat{\sigma}_{s}^{X},\sigma_{2i-2}^{Y}-\hat{\sigma}_{s}^{Y})\right\|^2\mathrm{d}s\right]\nonumber\\
		\leq&C_{u,v}\Delta_{n}^{2}.
	\end{align}
	Applying the first-order Taylor expansion for $\xi_{2i-1}^{(3,3)}(u,v)$, we obtain
	\begin{equation*}
		\mathbb{E}\left[\big|\xi_{2i-1}^{(3,3)}(u,v)\big|\Big|{\cal{G}}_{i-1}\right]
		\leq C_{u,v}\Delta_n^{-\frac12}\mathbb{E}\left[\int_{t_{2i-2}^n}^{t_{2i}^n}\big|(\hat{\sigma}_{s}^{X})^{2}-({\sigma}_{s}^{X})^{2}\big|+\big|(\hat{\sigma}_{s}^{Y})^{2}-({\sigma}_{s}^{Y})^{2}\big|\Big|{\cal{G}}_{i-1}\right],
	\end{equation*}
	consequently, we have
	\begin{equation}\label{estxi33}
		\mathbb{E}\left[\big|\xi_{2i-1}^{(3,3)}(u,v)\big|\right]
		\leq C_{u,v}\Delta_n^{3/2}.
	\end{equation}
%

	Next, for $\sum_{i=1}^{\lfloor n/2\rfloor}\xi_{2i-1}^{(1)}(u,v)$, we can decompose $\xi_{2i-1}^{(1)}(u,v)$ as $\sum_{i=1}^{5}\xi_{2i-1}^{(1,i)}(u,v)$, with
	\begin{align*}
		\xi_{2i-1}^{(1,1)}(u,v)=&-4\Delta_{n}^{1/2}{\rm{sin}}\left(\frac{1}{2\sqrt{\Delta_{n}}}\sqrt{2u}\Big(\Delta_{2i-1}^{n}X+\int_{t_{2i-2}^n}^{t_{2i-1}^n}b_{s}^{X}\mathrm{d}s+\int_{t_{2i-2}^n}^{t_{2i-1}^n}\sigma_{s}^{X}\mathrm{d}W_{s}^{X}\Big)\right.\\
		&+\left.\frac{1}{2\sqrt{\Delta_{n}}}\sqrt{2v}\Big(\Delta_{2i}^{n}Y+\int_{t_{2i-1}^n}^{t_{2i}^n}b_{s}^{Y}\mathrm{d}s+\int_{t_{2i-1}^n}^{t_{2i}^n}\sigma_{s}^{Y}\mathrm{d}W_{s}^{Y}\Big)\right)\\ &\times{\rm{sin}}\left(\frac{1}{2\sqrt{\Delta_{n}}}\sqrt{2u}\Big(\Delta_{2i-1}^{n}X-\int_{t_{2i-2}^n}^{t_{2i-1}^n}b_{s}^{X}\mathrm{d}s-\int_{t_{2i-2}^n}^{t_{2i-1}^n}\sigma_{s}^{X}\mathrm{d}W_{s}^{X}\Big)\right.\\
		&+\left.\frac{1}{2\sqrt{\Delta_{n}}}\sqrt{2v}\Big(\Delta_{2i}^{n}Y-\int_{t_{2i-1}^n}^{t_{2i}^n}b_{s}^{Y}\mathrm{d}s-\int_{t_{2i-1}^n}^{t_{2i}^n}\sigma_{s}^{Y}\mathrm{d}W_{s}^{Y}\Big)\right),\\
		\xi_{2i-1}^{(1,2)}(u,v)=&-2\sqrt{2u}\Delta_{n}{\rm{sin}}\left(\frac{\sqrt{2u}\sigma_{2i-2}^{X}\Delta_{2i-1}^{n}W^{X}+\sqrt{2v}\sigma_{2i-2}^{Y}\Delta_{2i}^{n}W^{Y}}{\sqrt{\Delta_{n}}}\right)b_{2i-2}^{X}\\
		&-2\sqrt{2v}\Delta_{n}{\rm{sin}}\left(\frac{\sqrt{2u}\sigma_{2i-2}^{X}\Delta_{2i-1}^{n}W^{X}+\sqrt{2v}\sigma_{2i-2}^{Y}\Delta_{2i}^{n}W^{Y}}{\sqrt{\Delta_{n}}}\right)b_{2i-1}^{Y},\\
		\xi_{2i-1}^{(1,3)}(u,v)=&-2{\rm{sin}}(\tilde{x}_{1}+\tilde{y}_{1})\left(\sqrt{2u}\int_{t_{2i-2}^n}^{t_{2i-1}^n}(b_{s}^{X}-b_{2i-2}^{X})\mathrm{d}s+\sqrt{2v}\int_{t_{2i-1}^n}^{t_{2i}^n}(b_{s}^{Y}-b_{2i-1}^{Y})\mathrm{d}s\right),\\		
		\xi_{2i-1}^{(1,4)}(u,v)=&-2\Delta_{n}^{-1/2}{\rm{cos}}(\tilde{x}_{2}+\tilde{y}_{2})\left(\sqrt{u}b_{2i-2}^{X}\Delta_{n}+\sqrt{u}\int_{t_{2i-2}^n}^{t_{2i-1}^n}(\sigma_{s}^{X}-\sigma_{2i-2}^{X})\mathrm{d}W_{s}^{X}\right.\\
		&\left.+\sqrt{v}b_{2i-1}^{Y}\Delta_{n}+\sqrt{v}\int_{t_{2i-1}^n}^{t_{2i}^n}(\sigma_{s}^{Y}-\sigma_{2i-2}^{Y})\mathrm{d}W_{s}^{Y}\right)^{2},\\
		\xi_{2i-1}^{(1,5)}(u,v)=&-2\sqrt{2u}{\rm{sin}}\left(\frac{\sqrt{2u}\sigma_{2i-2}^{X}\Delta_{2i-1}^{n}W^{X}+\sqrt{2v}\sigma_{2i-2}^{Y}\Delta_{2i}^{n}W^{Y}}{\sqrt{\Delta_{n}}}\right)\int_{t_{2i-2}^n}^{t_{2i-1}^n}(\sigma_{s}^{X}-\sigma_{2i-2}^{X})\mathrm{d}W_{s}^{Y}\\
		&-2\sqrt{2v}{\rm{sin}}\left(\frac{\sqrt{2u}\sigma_{2i-2}^{X}\Delta_{2i-1}^{n}W^{X}+\sqrt{2v}\sigma_{2i-2}^{Y}\Delta_{2i}^{n}W^{Y}}{\sqrt{\Delta_{n}}}\right)\int_{t_{2i-1}^n}^{t_{2i}^n}(\sigma_{s}^{Y}-\sigma_{2i-2}^{Y})\mathrm{d}W_{s}^{Y},
	\end{align*}
	where $\tilde{x}_{1}$ is between $\sqrt{2u}\Delta_{n}^{-1/2}\int_{t_{2i-2}^n}^{t_{2i-1}^n}b_{s}^{X}\mathrm{d}s+\sqrt{2u}\Delta_{n}^{-1/2}\int_{t_{2i-2}^n}^{t_{2i-1}^n}\sigma_{s}^{X}\mathrm{d}W_{s}^{X}$ and $\sqrt{2u}\Delta_{n}^{-1/2}b_{2i-2}^{X}\Delta_{n}+\int_{t_{2i-2}^n}^{t_{2i-1}^n}\sigma_{s}^{X}\mathrm{d}W_{s}^{X}$; \\
	$\tilde{y}_{1}$ is between $\sqrt{2v}\Delta_{n}^{-1/2}\int_{t_{2i-1}^n}^{t_{2i}^n}b_{s}^{Y}\mathrm{d}s+\sqrt{2v}\Delta_{n}^{-1/2}\int_{t_{2i-1}^n}^{t_{2i}^n}\sigma_{s}^{Y}\mathrm{d}W_{s}^{Y}$ and $\sqrt{2v}\Delta_{n}^{-1/2}b_{2i-1}^{Y}\Delta_{n}+\int_{t_{2i-1}^n}^{t_{2i}^n}\sigma_{s}^{Y}\mathrm{d}W_{s}^{Y}$;\\
	$\tilde{x}_{2}$ is between  $\sqrt{2u}\Delta_{n}^{-1/2}b_{2i-2}^{X}\Delta_{n}+\int_{t_{2i-2}^n}^{t_{2i-1}^n}\sigma_{s}^{X}\mathrm{d}W_{s}^{X}$
	and $\sqrt{2u}\Delta_{n}^{-1/2}\sigma_{2i-2}^{X}\Delta_{2i-1}^{n}W^{X}$;\\ $\tilde{y}_{2}$ is between
	$\sqrt{2v}\Delta_{n}^{-1/2}b_{2i-1}^{Y}\Delta_{n}+\int_{t_{2i-1}^n}^{t_{2i}^n}\sigma_{s}^{Y}\mathrm{d}W_{s}^{Y}$ and
	$\sqrt{2v}\Delta_{n}^{-1/2}\sigma_{2i-2}^{Y}\Delta_{2i}^{n}W^{Y}$.
	\\
	Following the inequalities $|{\rm{sin}}(A+B)|\leq|{\rm{sin}}(A)|+|{\rm{sin}}(B)|$, $|{\rm{sin}}(x)|\leq|x|$ and $|\sum_{i}|a_{i}||^{p}\leq\sum_{i}|a_{i}|^{p}$ for $0<p\leq1$, we have
	\begin{align}\label{estxi11}
		\mathbb{E}\left[\big|\xi_{2i-1}^{(1,1)}(u,v)\big|\right]\leq& C_{u,v}\Delta_{n}^{1/2-\beta/2-\iota/2}\left\{\mathbb{E}\left|\int_{t_{2i-2}^n}^{t_{2i-1}^n}\int_{{\mathbb{R}}}\delta^{X}(s-,x)\tilde{\mu}^{X}(\mathrm{d}s,\mathrm{d}x)\right|^{\beta+\iota}\right.\nonumber\\
		&+\left.\mathbb{E}\left|\int_{t_{2i-1}^n}^{t_{2i}^n}\int_{{\mathbb{R}}}\delta^{Y}(s-,y)\tilde{\mu}^{Y}(\mathrm{d}s,\mathrm{d}y)\right|^{\beta+\iota}\right\}\nonumber\\
		\leq& C_{u,v}\Delta_n^{3/2-\beta/2-\iota/2},
	\end{align}
	where $\iota\in(0,1-\beta)$.
	Applying the same techniques, we have
	\begin{align}
		&\mathbb{E}\left[\xi_{2i-1}^{(1,2)}(u,v)\big|{\cal{G}}_{n-1}\right]=0;\quad\mathbb{E}\left[\big|\xi_{2i-1}^{(1,2)}(u,v)\big|^{2}\right]\leq C_{u,v}\Delta_{n}^{2}, \label{estxi12}\\
		&\mathbb{E}\left[\big|\xi_{2i-1}^{(1,3)}(u,v)\big|\right]\leq C_{u,v}\Delta_{n}^{3/2},\label{estxi13}\\		
		&\mathbb{E}\left[\big|\xi_{2i-1}^{(1,4)}(u,v)\big|\right]\leq C_{u,v}\Delta_{n}^{3/2}.\label{estxi14}
	\end{align}
	For the remaining term $\xi_{2i-1}^{(1,5)}(u,v)$, we decompose it as $\sum_{i=1}^{2}\xi_{2i-1}^{(1,5,i)}(u,v)$, with
	\begin{align*}
		\xi_{2i-1}^{(1,5,1)}(u,v)=&-2f(\Delta_{2i-1}^{n}W^{X},\Delta_{2i}^{n}W^{Y})
		\left(\sqrt{2u}\int_{t_{2i-2}^n}^{t_{2i-1}^n}\zeta_{s}^{X}\mathrm{d}W_{s}^{X}
		+\sqrt{2v}\int_{t_{2i-1}^n}^{t_{2i}^n}\zeta_{s}^{Y}\mathrm{d}W_{s}^{Y}\right),\\
		\xi_{2i-1}^{(1,5,2)}(u,v)=&-2f(\Delta_{2i-1}^{n}W^{X},\Delta_{2i}^{n}W^{Y})
		\left(\sqrt{2u}\int_{t_{2i-2}^n}^{t_{2i-1}^n}\eta_{s}^{X}\mathrm{d}W_{s}^{X}
		+\sqrt{2v}\int_{t_{2i-1}^n}^{t_{2i}^n}\eta_{s}^{Y}\mathrm{d}W_{s}^{Y}\right),
	\end{align*}
    where
	\begin{align*}
		\zeta_{s}^{Z}=&\int_{t_{2i-2}^n}^{s}v_{2i-2}^{Z}{\rm{d}}W_{r}^{Z}+\int_{t_{2i-2}^n}^{s}v_{2i-2}^{'Z}{\rm{d}}W_{r}^{'Z}+\int_{t_{2i-2}^n}^{s}\int_{\mathbb{R}}\delta^{'Z}(t_{2i-2}^n-,z)\tilde{\mu}^{'Z}(\mathrm{d}r,\mathrm{d}z),\\
		\eta_{s}^{Z}=&\int_{t_{2i-2}^n}^{s}\tilde{b}_{r}^{Z}{\rm{d}}r+\int_{t_{2i-2}^n}^{s}\big(v_{r}^{Z}-v_{2i-2}^{Z}\big){\rm{d}}W_{r}^{Z}+\int_{t_{2i-2}^n}^{s}\big(v_{r}^{'Z}-v_{2i-2}^{'Z}\big){\rm{d}}W_{r}^{'Z}\\
		&+\int_{t_{2i-2}^n}^{s}\int_{\mathbb{R}}\big(\delta^{'Z}(r-,z)-\delta^{'Z}(t_{2i-2}^n-,z)\big)\tilde{\mu}^{'Z}(\mathrm{d}r,\mathrm{d}z),~~Z=X,Y,\\
		f(x,y)=&{\rm{sin}}\left(\frac{\sqrt{2u}\sigma_{2i-2}^{X}x+\sqrt{2v}\sigma_{2i-2}^{Y}y}{\sqrt{\Delta_{n}}}\right).
	\end{align*}
    Applying It\^{o}'s formula, we have
	\begin{align*}
		\mathbb{E}\left[f(\Delta_{2i-1}^{n}W^{X},\Delta_{2i}^{n}W^{Y})\int_{t_{2i-2}^n}^{t_{2i-1}^n}\left(\zeta_{s}^{X}-\int_{t_{2i-2}^n}^{s}v_{2i-2}^{X}{\rm{d}}W_{r}^{X}\right)\mathrm{d}W_{s}^{X}\Big|{\cal{G}}_{i-1}\right]=0,
	\end{align*}
	and
	\begin{align*}
		&\mathbb{E}\left[f(\Delta_{2i-1}^{n}W^{X},\Delta_{2i}^{n}W^{Y})\int_{t_{2i-2}^n}^{t_{2i-1}^n}\left(W_{s}^{X}-W_{2i-2}^{X}\right){\rm{d}}W_{s}^{X}\Big|{\cal{G}}_{i-1}\right]\\
		=&\frac{1}{2}\mathbb{E}\left[f(\Delta_{2i-1}^{n}W^{X},\Delta_{2i}^{n}W^{Y})\left(\left(\Delta_{2i-1}^{n}W^X\right)^{2}-\Delta_{n}\right)\Big|{\cal{G}}_{i-1}\right]=0.
	\end{align*}
    Similarly, we obtain
	\begin{align}
		&\mathbb{E}\left[\xi_{2i-1}^{(1,5,1)}(u,v)\big|{\cal{G}}_{i-1}\right]=0; \quad \mathbb{E}\left[\big|\xi_{2i-1}^{(1,5,1)}(u,v)\big|^{2}\right]\leq C_{u,v}\Delta_{n}^{2},\label{estxi151}\\
		&\mathbb{E}\left[\big|\xi_{2i-1}^{(1,5,2)}(u,v)\big|\right]\leq C_{u,v}\Delta_{n}^{3/2}.\label{estxi152}
	\end{align}

    Finally, it follows from \eqref{estxi31} to \eqref{estxi152}, that
    \begin{equation*}
	    \sum_{i=1}^{\lfloor n/2\rfloor}\xi_{2i-1}^{(1)}(u,v)\stackrel{P}{\longrightarrow}0,~~\sum_{i=1}^{\lfloor n/2\rfloor}\xi_{2i-1}^{(3)}(u,v)\stackrel{P}{\longrightarrow}0.
	\end{equation*}
    \qed
    \end{pf}

	Now, we will give proof of the main results.

	\begin{pot1}
		
	By the Lemma \ref{lemma1}, we have
	\begin{equation*}
		\sum_{i=1}^{\lfloor n/2\rfloor}\xi_{2i-1}^{(1)}(u,v)\stackrel{P}{\longrightarrow}0,~~\sum_{i=1}^{\lfloor n/2\rfloor}\xi_{2i-1}^{(3)}(u,v)\stackrel{P}{\longrightarrow}0.
	\end{equation*}
	Let $\zeta_{i}^{(2)}(u,v)=\xi_{2i-1}^{(2)}(u,v)$ and note that $\{\zeta_{j}^{(2)}(u,v),~j=1,2,\cdots,\lfloor n/2\rfloor\}$
	is a $\cal{G}$-martingale sequence. The argument used in the proof of Theorem 1 of  \cite{todorov2012realized} shows that we need to prove the finite-dimensional convergence of $\sum_{j=1}^{\lfloor n/2\rfloor}\zeta_j^{(2)}(u,v)$ and  the tightness of $\sum_{i=1}^{\lfloor n/2\rfloor}\sum_{j=1}^{3}\xi_{2i-1}^{(j)}(u,v)$ to complete the proof of the stable convergence in \eqref{equ3}.
	
	By a standard argument of the stable limit theorem, as stated in Theorem IX.7.28 of \cite{jacod2003limit}, it is enough to show the following results for the finite-dimensional convergence of $\sum_{j=1}^{\lfloor n/2\rfloor}\zeta_j^{(2)}(u,v)$:
	\begin{align}
		&\sum_{j=1}^{\lfloor n/2\rfloor}\mathbb{E}\left[\zeta_j^{(2)}(u,v)\zeta_j^{(2)}(u',v')\Big|{\cal G}_{j-1}\right]\stackrel{P}{\longrightarrow} F_p(u,v,u',v'),\label{asyco}\\
		&\sum_{j=1}^{\lfloor n/2\rfloor}\mathbb{E}\left[\left(\zeta_j^{(2)}(u,v)\right)^4\Big|{\cal G}_{j-1}\right]\stackrel{P}{\longrightarrow} 0,\label{fourmo}\\
		&\sum_{j=1}^{\lfloor n/2\rfloor}\mathbb{E}\left[\zeta_j^{(2)}(u,v)\delta_jW^Z\Big|{\cal G}_{j-1}\right]\stackrel{P}{\longrightarrow} 0, ~\text{for}~ Z=X, Y, \label{selfor}\\
		&\sum_{j=1}^{\lfloor n/2\rfloor}\mathbb{E}\left[\zeta_j^{(2)}(u,v)\delta_jN\Big|{\cal G}_{j-1}\right]\stackrel{P}{\longrightarrow} 0, \label{orth}
	\end{align}
	where $N$ is any bounded ${\cal G}$-martingale orthogonal to $W^X$ and $W^Y$, $\delta_jW^Z=W^Z_{2j\Delta_n}-W^Z_{2(j-1)\Delta_n}$ and
	$\delta_jN=N_{2j\Delta_n}-N_{2(j-1)\Delta_n}$.\\
	To show \eqref{asyco}, we have
	\begin{align*}
		&\sum_{j=1}^{\lfloor n/2\rfloor}\mathbb{E}\left[\zeta_j^{(2)}(u,v)\zeta_j^{(2)}(u',v')\Big|{\cal G}_{j-1}\right]\\
		=&2\Delta_n\sum_{j=1}^{\lfloor n/2\rfloor}\mathbb{E}\left[\cos\left(\frac{(\sqrt{2u}+\sqrt{2u'})\sigma_{2j-2}^X\Delta_{2j-1}^n{W^X}+(\sqrt{2v}+\sqrt{2v'})\sigma_{2j-2}^Y\Delta_{2j}^n{W^Y}}{\sqrt{\Delta_n}}\right)\Big|{\cal G}_{j-1}\right]\\
		&+2\Delta_n\sum_{j=1}^{\lfloor n/2\rfloor}\mathbb{E}\left[\cos\left(\frac{(\sqrt{2u}-\sqrt{2u'})\sigma_{2j-2}^X\Delta_{2j-1}^n{W^X}+(\sqrt{2v}-\sqrt{2v'})\sigma_{2j-2}^Y\Delta_{2j}^n{W^Y}}{\sqrt{\Delta_n}}\right)\Big|{\cal G}_{j-1}\right]\\
		&-4\Delta_n\sum_{j=1}^{\lfloor n/2\rfloor}\mbox{e}^{-\langle (u+u', v+v'), ((\sigma_{2j-2}^X)^2, (\sigma_{2j-2}^Y)^2)\rangle}\\
		=&2\Delta_n\sum_{j=1}^{\lfloor n/2\rfloor}\left\{\mbox{e}^{-(\sqrt{u}+\sqrt{u'})^2(\sigma_{2j-2}^X)^2-(\sqrt{v}+\sqrt{v'})^2(\sigma_{2j-2}^Y)^2}+\mbox{e}^{-(\sqrt{u}-\sqrt{u'})(\sigma_{2j-2}^X)^2-(\sqrt{v}-\sqrt{v'})(\sigma_{2j-2}^Y)^2}\right\}\\
		&-4\Delta_n\sum_{i=1}^{\lfloor n/2\rfloor}\mbox{e}^{-\langle (u+u', v+v'), ((\sigma_{2j-2}^X)^2, (\sigma_{2j-2}^Y)^2)\rangle}\\
		\stackrel{P}{\longrightarrow}&\int_0^T\left\{\mbox{e}^{-(\sqrt{u}+\sqrt{u'})^2(\sigma_s^X)^2-(\sqrt{v}+\sqrt{v'})^2(\sigma_s^Y)^2}+\mbox{e}^{-(\sqrt{u}-\sqrt{u'})^2(\sigma_s^X)^2-(\sqrt{v}-\sqrt{v'})^2(\sigma_s^Y)^2}\right.\\
		&\left.~~~~~~-2\mbox{e}^{-(u+u')(\sigma_s^X)^2-(v+v')(\sigma_s^Y)^2}\right\}{\rm{d}}s.
	\end{align*}
	For \eqref{fourmo}, we simply have
	\begin{equation*}
		\sum_{j=1}^{\lfloor n/2\rfloor}\mathbb{E}\left[\left(\zeta_j^{(2)}(u,v)\right)^4\Big|{\cal G}_{j-1}\right]\leq C\Delta_n\stackrel{P}{\longrightarrow}0.
	\end{equation*}
	Next we show \eqref{selfor} by letting $Z=X$, since the proof for $Z=Y$ is similar. Observing that
	\begin{align}\label{zetadeltaw}
		&\sum_{j=1}^{\lfloor n/2\rfloor}\mathbb{E}\left[\zeta_j^{(2)}(u,v)\delta_jW^X\Big|{\cal G}_{j-1}\right]\nonumber\\
		=&2\Delta_n^{\frac12}\sum_{j=1}^{\lfloor n/2\rfloor}\mathbb{E}\left[\cos\left(\frac{\sqrt{2u}\sigma_{2j-2}^X\Delta_{2j-1}^n{W^X}+\sqrt{2v}\sigma_{2j-2}^Y\Delta_{2j}^n{W^Y}}{\sqrt{\Delta_n}}\right)\delta_jW^X\Big|{\cal G}_{j-1}\right]
		\nonumber\\
		=&2\Delta_n^{\frac12}\sum_{j=1}^{\lfloor n/2\rfloor}\mathbb{E}\left[\cos\left(\frac{\sqrt{2u}\sigma_{2j-2}^X\Delta_{2j-1}^n{W^X}+\sqrt{2v}\sigma_{2j-2}^Y\Delta_{2j}^n{W^Y}}{\sqrt{\Delta_n}}\right)\Delta_{2j-1}^nW^X\Big|{\cal G}_{j-1}\right]\nonumber\\
		&+2\Delta_n^{\frac12}\sum_{j=1}^{\lfloor n/2\rfloor}\mathbb{E}\left[\cos\left(\frac{\sqrt{2u}\sigma_{2j-2}^X\Delta_{2j-1}^n{W^X}+\sqrt{2v}\sigma_{2j-2}^Y\Delta_{2j}^n{W^Y}}{\sqrt{\Delta_n}}\right)\Delta_{2j}^nW^X\Big|{\cal G}_{j-1}\right],
	\end{align}
	the first term of \eqref{zetadeltaw} is equal to zero because the variable inside the expectation is an odd function of $\Delta_{2j-1}^nW^X$ and $\Delta_{2j}^nW^Y$ is independent of $\Delta_{2j-1}^nW^X$. For the second term of \eqref{zetadeltaw}, we have
	\begin{align}\label{zetadelta2}
		&\sum_{j=1}^{\lfloor n/2\rfloor}\mathbb{E}\left[\cos\left(\frac{\sqrt{2u}\sigma_{2j-2}^X\Delta_{2j-1}^n{W^X}+\sqrt{2v}\sigma_{2j-2}^Y\Delta_{2j}^n{W^Y}}{\sqrt{\Delta_n}}\right)\Delta_{2j}^nW^X\Big|{\cal G}_{j-1}\right]\nonumber\\
		=&\sum_{j=1}^{\lfloor n/2\rfloor}\mathbb{E}\left[\cos\left(\frac{\sqrt{2u}\sigma_{2j-2}^X\Delta_{2j-1}^n{W^X}}{\sqrt{\Delta_n}}\right)\cos\left(\frac{\sqrt{2v}\sigma_{2j-2}^Y\Delta_{2j}^n{W^Y}}{\sqrt{\Delta_n}}\right)\Delta_{2j}^nW^X\Big|{\cal G}_{j-1}\right]\nonumber\\
		&-\sum_{j=1}^{\lfloor n/2\rfloor}\mathbb{E}\left[\sin\left(\frac{\sqrt{2u}\sigma_{2j-2}^X\Delta_{2j-1}^n{W^X}}{\sqrt{\Delta_n}}\right)\sin\left(\frac{\sqrt{2v}\sigma_{2j-2}^Y\Delta_{2j}^n{W^Y}}{\sqrt{\Delta_n}}\right)\Delta_{2j}^nW^X\Big|{\cal G}_{j-1}\right].
	\end{align}
	The second term of \eqref{zetadelta2} is zero because $\sin(\cdot)$ is an odd function and $\Delta_{2j-1}^n{W^X}$ is independent of both $\Delta_{2j}^n{W^X}$ and $\Delta_{2j}^n{W^Y}$. For the first term of \eqref{zetadelta2}, It\^{o}'s formula yields
	\begin{align*}
		\cos\left(\frac{\sqrt{2v}\sigma_{2j-2}^Y\Delta_{2j}^n{W^Y}}{\sqrt{\Delta_n}}\right)
		=&1+\int_{t_{2j-1}^n}^{t_{2j}^n}\sin\left(\frac{\sqrt{2v}\sigma_{2j-2}^Y\int_{t_{2j-1}^n}^s \mathrm{d}W_u^Y}{\sqrt{\Delta_n}}\right)
		\frac{\sqrt{2v}}{\sqrt{\Delta_n}}\sigma_{2j-2}^Y \mathrm{d}W_s^Y\\
		&-\frac{1}{2}\int_{t_{2j-1}^n}^{t_{2j}^n}\cos\left(\frac{\sqrt{2v}\sigma_{2j-2}^X\int_{t_{2j-1}^n}^s \mathrm{d}W_u^Y}{\sqrt{\Delta_n}}\right)
		\left(\frac{\sqrt{2v}}{\sqrt{\Delta_n}}\sigma_{2j-2}^Y\right)^2\mathrm{d}s.
	\end{align*}
	Hence, by product formula, we obtain
	\begin{align*}
		&\mathbb{E}\left[\cos\left(\frac{\sqrt{2v}\sigma_{2j-2}^Y\Delta_{2j}^n{W^Y}}{\sqrt{\Delta_n}}\right)\Delta_{2j}^nW^X\Big|{\cal G}_{j-1}\right]\\
		=& \mathbb{E}\left[\int_{t_{2j-1}^n}^{t_{2j}^n}\cos\left(\frac{\sqrt{2v}\sigma_{2j-2}^X\int_{t_{2j-1}^n}^s \mathrm{d}W_u^Y}{\sqrt{\Delta_n}}\right)
		\left(\frac{\sqrt{2v}}{\sqrt{\Delta_n}}\sigma_{2j-2}^Y\right)^2\int_{t_{2j-1}^n}^s \mathrm{d}W_u^X\mathrm{d}s\Big|{\cal G}_{j-1}\right].
	\end{align*}
	Now, we define a function
	$$f(s):=\mathbb{E}\left[\cos\left(\frac{\sqrt{2v}\sigma_{2j-2}^Y\int_{t_{2j-1}^n}^s \mathrm{d}W_u^Y}{\sqrt{\Delta_n}}\right)\int_{t_{2j-1}^n}^s \mathrm{d}W_u^X\Bigg|{\cal G}_{j-1}\right]$$
	for $s\in[t_{2j-1}^n, t_{2j}^n]$, then we have $f(t_{2j-1}^n)=0$ and $f(s)=c\int_{t_{2j-1}^n}^sf(u)du$ where $c$ is a constant. It is easy to see that $f(s)$ is continuous, then $f(s)\equiv0$ in view of the initial condition. Therefore, we have
	\begin{align*}
		&\sum_{j=1}^{\lfloor n/2\rfloor}\mathbb{E}\left[\cos\left(\frac{\sqrt{2u}\sigma_{2j-2}^X\Delta_{2j-1}^n{W^X}}{\sqrt{\Delta_n}}\right)\cos\left(\frac{\sqrt{2v}\sigma_{2j-2}^Y\Delta_{2j}^n{W^Y}}{\sqrt{\Delta_n}}\right)\Delta_{2j}^nW^X\Big|{\cal G}_{j-1}\right]\\
		=&\sum_{j=1}^{\lfloor n/2\rfloor}\mathbb{E}\left[\cos\left(\frac{\sqrt{2u}\sigma_{2j-2}^X\Delta_{2j-1}^n{W^X}}{\sqrt{\Delta_n}}\right)\Big|{\cal G}_{j-1}\right]\times\\
		& \mathbb{E}\left[\int_{t_{2j-2}^n}^{t_{2j}^n}\cos\left(\frac{\sqrt{2v}\sigma_{2j-2}^X\int_{t_{2j}^n}^s \mathrm{d}W_u^Y}{\sqrt{\Delta_n}}\right)
		\left(\frac{\sqrt{2v}}{\sqrt{\Delta_n}}\sigma_{2j-2}^Y\right)^2\int_{t_{2j}^n}^s \mathrm{d}W_u^X\mathrm{d}s\Big|{\cal G}_{j-1}\right]\\
		=&0.
	\end{align*}
	Finally, for \eqref{orth}, since $N$ is orthogonal to $W^X$ and $W^Y$ and defined on ${\cal F}$, we have
	$$\mathbb{E}\left[\zeta_j^{(2)}(u,v)\delta_jN\Big|{\cal G}_{j-1}\right]=0,$$
	which yields the desired result.
	
	We now show the tightness. For the convenience of the following explanation, we denote:
	\begin{align*}
		\hat{S}_{t,1}(u,v)=&\sum_{i=\lfloor(t-1)/2\Delta_{n}\rfloor+1}^{\lfloor t/2\Delta_{n}\rfloor}\xi_{2i-1}^{(2)}(u,v), \\
		\hat{S}_{t,2}(u,v)=&\sum_{i=\lfloor(t-1)/2\Delta_{n}\rfloor+1}^{\lfloor t/2\Delta_{n}\rfloor}\left(\xi_{2i-1}^{(3,2)}(u,v)+\xi_{2i-1}^{(1,2)}(u,v)+\xi_{2i-1}^{(1,5,1)}(u,v)\right),\\
		\hat{S}_{t,3}(u,v)=&\sum_{i=\lfloor(t-1)/2\Delta_{n}\rfloor+1}^{\lfloor t/2\Delta_{n}\rfloor}\left(\xi_{2i-1}^{(1)}(u,v)+\xi_{2i-1}^{(3)}(u,v)\right)-\hat{S}_{t,2}(u,v).
	\end{align*}
	For $u_{1},u_{2},v_{1},v_{2}\in\mathbb{R}_{+}$,
	\begin{align}
		&\mathbb{E}\left[\sum_{t=1}^{T}\left(\hat{S}_{t,1}(u_{2},v_{2})+\hat{S}_{t,2}(u_{2},v_{2})\right)\right]^{2}\leq C, \label{equation19}\\
		\nonumber&\mathbb{E}\left[\sum_{t=1}^{T}\left(\hat{S}_{t,1}(u_{1},v_{1})+\hat{S}_{t,2}(u_{1},v_{1})-\hat{S}_{t,1}(u_{2},v_{2})-\hat{S}_{t,2}(u_{2},v_{2})\right)\right]^2\\
		\leq&C\left[(\sqrt{u_{1}}-\sqrt{u_{2}})^{2}+(\sqrt{v_{1}}-\sqrt{v_{2}})^{2}\right]\vee\left[(u_{1}-u_{2})^{2}+(v_{1}-v_{2})^{2}\right].\label{equation20}
	\end{align}
	According to Theorem 20 of \cite{ibragimov1981statistical}, we have \\ $\sum_{t=1}^{T}\left(\hat{S}_{t,1}(u,v)+\hat{S}_{t,2}(u,v)\right)$ is tight. Since $$\Delta_{n}^{\beta/2+\iota/2-1/2}\left|\sup\limits_{0\leq u\leq\bar{u},0\leq v\leq\bar{v}}\sum_{i=1}^{\lfloor T/2\Delta_{n}\rfloor}\hat{S}_{t,3}(u,v)\right|$$
	is bounded in probability. Therefore, $\sum_{i=1}^{\lfloor n/2\rfloor}\sum_{j=1}^{3}\xi_{2i-1}^{(j)}(u,v)$ is tight.
	
    Finally, we prove the $\tilde{\Gamma}_{n}$ is a consistent estimator:
	\begin{align*}
		{\tilde{\Gamma}_n}=&{4\Delta_{n}\sum_{i=1}^{\lfloor n/2\rfloor}\mathrm{cos}\big(\frac{\sqrt{2u}\Delta_{2i-1}^{n}X+\sqrt{2v}\Delta_{2i}^{n}Y}{\sqrt{\Delta_{n}}}\big)\mathrm{cos}\big(\frac{\sqrt{2u'}\Delta_{2i-1}^{n}X+\sqrt{2v'}\Delta_{2i}^{n}Y}{\sqrt{\Delta_{n}}}\big)}\nonumber\\
		&{-4\Delta_{n}\sum_{i=1}^{\lfloor n/2\rfloor}\mathrm{cos}\big(\frac{\sqrt{2(u+u')}\Delta_{2i-1}^{n}X+\sqrt{2(v+v')}\Delta_{2i}^{n}Y}{\sqrt{\Delta_{n}}}\big)}.\nonumber\\
		=&4\sum_{i=1}^{\lfloor n/2\rfloor}\frac{1}{2}\Delta_{n}\left[\mathrm{cos}\big(\frac{(\sqrt{2u}+\sqrt{2u'})\Delta_{2i-1}^{n}X+(\sqrt{2v}+\sqrt{2v'})\Delta_{2i}^{n}Y}{\sqrt{\Delta_{n}}}\big)\right.\nonumber\\
		&\left.+\mathrm{cos}\big(\frac{(\sqrt{2u}-\sqrt{2u'})\Delta_{2i-1}^{n}X+(\sqrt{2v}-\sqrt{2v'})\Delta_{2i}^{n}Y}{\sqrt{\Delta_{n}}}\big)\right]\nonumber\\
		&-4\sum_{i=1}^{\lfloor n/2\rfloor}\Delta_{n}\mathrm{cos}\big(\frac{\sqrt{2(u+u')}\Delta_{2i-1}^{n}X+\sqrt{2(v+v')}\Delta_{2i}^{n}Y}{\sqrt{\Delta_{n}}}\big).\nonumber\\
	\end{align*}
Because of the Law of Large Numbers,
\begin{align*}
&{\tilde{\Gamma}_n}-\left[2\Delta_{n}\sum_{i=1}^{\lfloor n/2\rfloor}e^{-(\sqrt{u}+\sqrt{u'})^2(\sigma_{2i-2}^{X})^2-(\sqrt{v}+\sqrt{v'})^2(\sigma_{2i-2}^{Y})^2}+2\Delta_{n}\sum_{i=1}^{\lfloor n/2\rfloor}e^{-(\sqrt{u}-\sqrt{u'})^2(\sigma_{2i-2}^{X})^2-(\sqrt{v}-\sqrt{v'})^2(\sigma_{2i-2}^{Y})^2}\right. \nonumber \\
		&\left.-4\Delta_{n}\sum_{i=1}^{\lfloor n/2\rfloor}e^{-(u+u')(\sigma_{2i-2}^{X})^2-(v+v')(\sigma_{2i-2}^{Y})^2}\right]\stackrel{P}{\longrightarrow}0,\nonumber\\
\end{align*}
and
\begin{align*}
&2\Delta_{n}\sum_{i=1}^{\lfloor n/2\rfloor}e^{-(\sqrt{u}+\sqrt{u'})^2(\sigma_{2i-2}^{X})^2-(\sqrt{v}+\sqrt{v'})^2(\sigma_{2i-2}^{Y})^2}+2\Delta_{n}\sum_{i=1}^{\lfloor n/2\rfloor}e^{-(\sqrt{u}-\sqrt{u'})^2(\sigma_{2i-2}^{X})^2-(\sqrt{v}-\sqrt{v'})^2(\sigma_{2i-2}^{Y})^2} \nonumber \\
		&-4\Delta_{n}\sum_{i=1}^{\lfloor n/2\rfloor}e^{-(u+u')(\sigma_{2i-2}^{X})^2-(v+v')(\sigma_{2i-2}^{Y})^2}\nonumber\\
\stackrel{P}{\longrightarrow}&\int_0^T\left\{\mbox{e}^{-(\sqrt{u}+\sqrt{u'})^2(\sigma_s^X)^2-(\sqrt{v}+\sqrt{v'})^2(\sigma_s^Y)^2}+\mbox{e}^{-(\sqrt{u}-\sqrt{u'})^2(\sigma_s^X)^2-(\sqrt{v}-\sqrt{v'})^2(\sigma_s^Y)^2}\right.\nonumber\\
		&\left.~~~~~~-2\mbox{e}^{-(u+u')(\sigma_s^X)^2-(v+v')(\sigma_s^Y)^2}\right\}{\rm{d}}s.\nonumber
\end{align*}
We can achieve
\begin{align*}
{\tilde{\Gamma}_n}\stackrel{P}{\longrightarrow}&\int_0^T\left\{\mbox{e}^{-(\sqrt{u}+\sqrt{u'})^2(\sigma_s^X)^2-(\sqrt{v}+\sqrt{v'})^2(\sigma_s^Y)^2}+\mbox{e}^{-(\sqrt{u}-\sqrt{u'})^2(\sigma_s^X)^2-(\sqrt{v}-\sqrt{v'})^2(\sigma_s^Y)^2}\right.\nonumber\\
		&\left.~~~~~~-2\mbox{e}^{-(u+u')(\sigma_s^X)^2-(v+v')(\sigma_s^Y)^2}\right\}{\rm{d}}s.\nonumber
\end{align*}
	\end{pot1}
	

	\begin{pot2}
	We redefine
	\begin{align*}
		\xi_{i}^{(1)}(u,v)=&\Delta_n^{\frac12}\left\{\cos\big(\frac{\sqrt{2u}\Delta_i^n{X}+\sqrt{2v}\Delta_{i+1}^n{Y}}{\sqrt{\Delta_n}}\big)- \cos\big(\frac{\sqrt{2u}\sigma_{i-1}^X\Delta_i^n{W^X}+\sqrt{2v}\sigma_{i-1}^Y\Delta_{i+1}^n{W^Y}}{\sqrt{\Delta_n}}\big)\right\},\\
		\xi_{i}^{(2)}(u,v)=&\Delta_n^{\frac12}\left\{\cos\big(\frac{\sqrt{2u}\sigma_{i-1}^X\Delta_i^n{W^X}+\sqrt{2v}\sigma_{i-1}^Y\Delta_{i+1}^n{W^Y}}{\sqrt{\Delta_n}}\big)-\mbox{e}^{-\langle (u, v), ((\sigma_{i-1}^X)^2, (\sigma_{i-1}^Y)^2)\rangle}\right\},\\
		\xi_{i}^{(3)}(u,v)=&\Delta_n^{-\frac12}\left\{\Delta_n\mbox{e}^{-\langle (u, v), ((\sigma_{i-1}^X)^2, (\sigma_{i-1}^Y)^2)\rangle}-\int_{(i-1)\Delta_n}^{i\Delta_n}\mbox{e}^{-\langle (u, v),((\sigma_s^X)^2, (\sigma_s^Y)^2)\rangle}\mathrm{d}s\right\}.
	\end{align*}
	By the same techniques used in the proof of Lemma \ref{lemma1}, we can show that
	\begin{eqnarray}\label{equation21}
		\sum_{i=1}^{n-1}\xi_{i}^{(1)}(u,v)\stackrel{P}{\longrightarrow}0,~~\sum_{i=1}^{n-1}\xi_{i}^{(3)}(u,v)\stackrel{P}{\longrightarrow}0.
	\end{eqnarray}
	Similar to the proof of Theorem \ref{thm1}, to obtain the convergence in \eqref{conve2}, we only need to show the finite-dimensional convergence of $\sum_{i=1}^{n-1}\xi_i^{(2)}(u,v)$ and the tightness of $\sum_{i=1}^{n-1}\sum_{j=1}^{3}\xi_{i}^{(j)}(u,v)$.
	
	First consider $\sum_{i=1}^{n-1}\xi_{i}^{(2)}(u,v)$. To deal with the dependence of increments, we use the big-small-block method. Finally, those big blocks dominate the asymptotic performance, while those small blocks are negligible. To create, we first split the increments so that each block contains $p+1$ increments, and then gather the first $p$ increments (big block), and the last one belongs to another block (small block). Precisely, let $N_p=\lfloor\frac{n-1}{p+1}\rfloor$, $\alpha_j^p:=(j-1)(p+1)$, and
	\begin{equation*}
		\zeta_j^{n, p}(u,v)=\sum_{i=\alpha_{j}^p+1}^{\alpha_{j+1}^p-1}\xi_i^{(2)}(u,v),~\eta_j^{n,p}(u,v)=\xi_{\alpha_{j+1}^p}^{(2)}(u,v).
	\end{equation*}
	Hence,
	\begin{equation*}
		\sum_{i=1}^{n-1}\xi_{i}^{(2)}(u,v)=\sum_{j=1}^{N_p}\zeta_j^{n, p}(u,v)+\sum_{j=1}^{N_p}\eta_j^{n, p}(u,v)+\sum_{i=N_p(p+1)+1}^{n-1}\xi_{i}^{(2)}(u,v).
	\end{equation*}
	Observing that $\sup_{u,v}\sup_i|\xi_{i}^{(2)}(u,v)|\leq 2\Delta_n^{1/2}$, thus
	\begin{align*}
		&\limsup_{p\rightarrow\infty}\lim_{n\rightarrow\infty}\sup_{u,v}\mathbb{E}\left|\sum_{j=1}^{N_p}\eta_j^{n, p}(u,v)+\sum_{i=N_p(p+1)+1}^{n-1}\xi_{i}^{(2)}(u,v)\right|^{2}\\
		\leq& C\limsup_{p\rightarrow\infty}\lim_{n\rightarrow\infty}\Big(\frac{1}{p+1}+p\Delta_{n}\Big)=0.
	\end{align*}
	For the term $\sum_{j=1}^{N_p}\zeta_j^{n, p}(u,v)$, we consider $\sum_{j=1}^{N_p}\zeta_j^{n, p\prime}(u,v)$ with
	\begin{align*}
		&\zeta_j^{n, p\prime}(u,v)=\sum_{i=\alpha_j^p+1}^{\alpha_{j+1}^p-1}\xi_i^{(2)\prime}(u,v),\\
		&\xi_{i}^{(2)\prime}(u,v)=\Delta_n^{\frac12}\left\{\cos\Big(\frac{\sqrt{2u}\sigma_{\alpha_j^p}^X\Delta_i^n{W^X}+\sqrt{2v}\sigma_{\alpha_j^p}^Y\Delta_{i+1}^n{W^Y}}{\sqrt{\Delta_n}}\Big)-\mbox{e}^{-\langle (u, v), ((\sigma_{\alpha_j^p}^X)^2, (\sigma_{\alpha_j^p}^Y)^2)\rangle}\right\}.
	\end{align*}
	Under Assumption \ref{asu2}, we can show that
	\begin{equation}\label{equation22}
		\sum_{j=1}^{N_p}\Big(\zeta_j^{n, p}(u,v)-\zeta_j^{n, p\prime}(u,v)\Big)\stackrel{P}{\longrightarrow}0.
	\end{equation}
	The proof of the convergence in \eqref{equation22} is similar to that in \eqref{equation21}.
	Denote ${\cal G}_j:={\cal F}_{\alpha_j^p}$, then $\{\zeta_j^{n, p\prime}(u,v), ~j=1,2,\cdots,N_p\}$ is a ${\cal G}$-martingale sequence. In order to apply the Theorm IX.7.28 of \cite{jacod2003limit}, it is sufficient to show the following convergence results:
	\begin{align}
		&\sum_{j=1}^{N_p}\mathbb{E}[\zeta_j^{n, p\prime}(u,v)\zeta_j^{n, p\prime}(u',v')|{\cal G}_j]\stackrel{P}{\longrightarrow} \int_0^TF(\sqrt{u}\sigma_s^X,\sqrt{v}\sigma_s^Y,\sqrt{u'}\sigma_s^X,\sqrt{v'}\sigma_s^Y, \rho_s)\mathrm{d}s,\label{asyco1}\\
		&\sum_{j=1}^{N_p}\mathbb{E}[\big(\zeta_j^{n, p\prime}(u,v)\big)^4|{\cal G}_j]\stackrel{P}{\longrightarrow} 0,\label{fourmo2}\\
		&\sum_{j=1}^{N_p}\mathbb{E}[\zeta_j^{n, p\prime}(u,v)\Delta_jW^Z|{\cal G}_j]\stackrel{P}{\longrightarrow} 0, ~\text{for}~ Z=X, Y, \label{selfor2}\\
		&\sum_{j=1}^{N_p}\mathbb{E}[\zeta_j^{n, p\prime}(u,v)\Delta_jN|{\cal G}_j]\stackrel{P}{\longrightarrow} 0, \label{orth2}\\
		&\lim_{p\rightarrow\infty}F_p(\sqrt{u}\sigma_s^X,\sqrt{v}\sigma_s^Y,\sqrt{u'}\sigma_s^X,\sqrt{v'}\sigma_s^Y, \rho_s)=F(\sqrt{u}\sigma_s^X,\sqrt{v}\sigma_s^Y,\sqrt{u'}\sigma_s^X,\sqrt{v'}\sigma_s^Y, \rho_s),\nonumber
	\end{align}
	where $N$ is any bounded ${\cal G}$-martingale orthogonal to $W^X$ and $W^Y$,
	\begin{align*}
		F_p(\sqrt{u}\sigma_s^X,\sqrt{v}\sigma_s^Y,\sqrt{u'}\sigma_s^X,\sqrt{v'}\sigma_s^Y,\rho_s)=&\frac{p}{p+1}F_{1}(\sqrt{u}\sigma_s^X,\sqrt{v}\sigma_s^Y,\sqrt{u'}\sigma_s^X,\sqrt{v'}\sigma_s^Y,\rho_s)\\
		&+\frac{p-1}{p+1}F_{2}(\sqrt{u}\sigma_s^X,\sqrt{v}\sigma_s^Y,\sqrt{u'}\sigma_s^X,\sqrt{v'}\sigma_s^Y,\rho_s)\\
		&+\frac{p-1}{p+1}F_{3}(\sqrt{u}\sigma_s^X,\sqrt{v}\sigma_s^Y,\sqrt{u'}\sigma_s^X,\sqrt{v'}\sigma_s^Y,\rho_s),
	\end{align*}
	and the definition of $F_1$, $F_2$ and $F_3$ will be given in the following proof.
	Note that the proofs of \eqref{fourmo2} to \eqref{orth2} are similar to those of \eqref{fourmo} to \eqref{orth}. Thus, we only give the proof of \eqref{asyco1}.
	To show \eqref{asyco1}, note that
	\begin{align*}
		\mathbb{E}[\zeta_j^{n, p\prime}(u,v)\zeta_j^{n, p\prime}(u',v')|{\cal G}_j]=&\sum_{i=\alpha_j^p+1}^{\alpha_{j+1}^p-1}\mathbb{E}[\xi_i^{(2)\prime}(u,v)\xi_i^{(2)\prime}(u',v')|{\cal G}_j]\\
		&+\sum_{i=\alpha_j^p+1}^{\alpha_{j+1}^p-2}\mathbb{E}[\xi_i^{(2)\prime}(u,v)\xi_{i+1}^{(2)\prime}(u',v')|{\cal G}_j]\\
		&+\sum_{i=\alpha_j^p+1}^{\alpha_{j+1}^p-2}\mathbb{E}[\xi_i^{(2)\prime}(u',v')\xi_{i+1}^{(2)\prime}(u,v)|{\cal G}_j].
	\end{align*}
	First,
	\begin{align*}
		&~~~~\sum_{i=\alpha_j^p+1}^{\alpha_{j+1}^p-1}\mathbb{E}[\xi_i^{(2)\prime}(u,v)\xi_i^{(2)\prime}(u',v')|{\cal G}_j]\\
		&=\frac{p\Delta_n}{2}\mbox{e}^{-u(\sigma_{\alpha_j^p}^X)^2-v(\sigma_{\alpha_j^p}^Y)^2-u'(\sigma_{\alpha_j^p}^X)^2-v'(\sigma_{\alpha_j^p}^Y)^2}\times\\
		&~~~~~~~~~~~~~~~\{\mbox{e}^{-2((\sqrt{uu'}\sigma_{\alpha_j^p}^X)^2+(\sqrt{vv'}\sigma_{\alpha_j^p}^Y)^2)}+\mbox{e}^{2((\sqrt{uu'}\sigma_{\alpha_j^p}^X)^2+(\sqrt{vv'}\sigma_{\alpha_j^p}^Y)^2)}-2\}\\
		&=:\frac{p\Delta_n}{2}F_{1}(\sqrt{u}\sigma_{\alpha_j^p}^X,\sqrt{v}\sigma_{\alpha_j^p}^Y,\sqrt{u'}\sigma_{\alpha_j^p}^X,\sqrt{v'}\sigma_{\alpha_j^p}^Y, \rho_{\alpha_j^p}).
	\end{align*}
	Second,
	\begin{align*}
		&~~~~\sum_{i=\alpha_j^p+1}^{\alpha_{j+1}^p-2}\mathbb{E}[\xi_i^{(2)\prime}(u,v)\xi_{i+1}^{(2)\prime}(u',v')|{\cal G}_j]\\
		&=\frac{(p-1)\Delta_n}{2}\mbox{e}^{-u(\sigma_{\alpha_j^p}^X)^2-v(\sigma_{\alpha_j^p}^Y)^2-u'(\sigma_{\alpha_j^p}^X)^2-v'(\sigma_{\alpha_j^p}^Y)^2}\{\mbox{e}^{-2\sqrt{u'v}\sigma_{\alpha_j^p}^X\sigma_{\alpha_j^p}^Y\rho_{\alpha_j^p}}+\mbox{e}^{2\sqrt{u'v}\sigma_{\alpha_j^p}^X\sigma_{\alpha_j^p}^Y\rho_{\alpha_j^p}}-2\}\\
		&=:\frac{(p-1)\Delta_n}{2}F_{2}(\sqrt{u}\sigma_{\alpha_j^p}^X,\sqrt{v}\sigma_{\alpha_j^p}^Y,\sqrt{u'}\sigma_{\alpha_j^p}^X,\sqrt{v'}\sigma_{\alpha_j^p}^Y, \rho_{\alpha_j^p}).
	\end{align*}
	Third,
	\begin{align*}
		&~~~~\sum_{i=\alpha_j^p+1}^{\alpha_{j+1}^p-2}\mathbb{E}[\xi_i^{(2)\prime}(u',v')\xi_{i+1}^{(2)\prime}(u,v)|{\cal G}_j]\\
		&=\frac{(p-1)\Delta_n}{2}\mbox{e}^{-u(\sigma_{\alpha_j^p}^X)^2-v(\sigma_{\alpha_j^p}^Y)^2-u'(\sigma_{\alpha_j^p}^X)^2-v'(\sigma_{\alpha_j^p}^Y)^2}\{\mbox{e}^{-2\sqrt{uv'}\sigma_{\alpha_j^p}^X\sigma_{\alpha_j^p}^Y\rho_{\alpha_j^p}}+\mbox{e}^{2\sqrt{uv'}\sigma_{\alpha_j^p}^X\sigma_{\alpha_j^p}^Y\rho_{\alpha_j^p}}-2\}\\
		&=:\frac{(p-1)\Delta_n}{2}F_{3}(\sqrt{u}\sigma_{\alpha_j^p}^X,\sqrt{v}\sigma_{\alpha_j^p}^Y,\sqrt{u'}\sigma_{\alpha_j^p}^X,\sqrt{v'}\sigma_{\alpha_j^p}^Y, \rho_{\alpha_j^p}).
	\end{align*}
	Then, we can obtain \eqref{asyco1}.
	Next, the tightness argument in Theorem \ref{thm1} also proves the tightness of $\sum_{i=1}^{n-1}\sum_{j=1}^{3}\xi_{i}^{(j)}(u,v)$. Hence, the proof of the convergence in \eqref{conve2} is complete.
	
	Finally, the proof of the consistency of $\tilde{\Gamma}_{n}$ is showed as follows:
	\begin{align*}
		{\tilde{\Gamma}_n}=&\Delta_{n}\sum_{i=1}^{n-1}\mathrm{cos}\big(\frac{\sqrt{2u}\Delta_{i}^{n}X+\sqrt{2v}\Delta_{i+1}^{n}Y}{\sqrt{\Delta_{n}}}\big)\mathrm{cos}\big(\frac{\sqrt{2u'}\Delta_{i}^{n}X+\sqrt{2v'}\Delta_{i+1}^{n}Y}{\sqrt{\Delta_{n}}}\big)\nonumber\\ &+\Delta_{n}\sum_{i=1}^{n-2}\mathrm{cos}\big(\frac{\sqrt{2u}\Delta_{i-1}^{n}X}{\sqrt{\Delta_{n}}}\big) \mathrm{cos}\big(\frac{\sqrt{2v}\Delta_{i+1}^{n}Y}{\sqrt{\Delta_{n}}}\big)\mathrm{cos}\big(\frac{\sqrt{2u'}\Delta_{i+1}^{n}X}{\sqrt{\Delta_{n}}}\big)\mathrm{cos}\big(\frac{\sqrt{2v'}\Delta_{i+2}^{n}Y}{\sqrt{\Delta_{n}}}\big)\nonumber\\
		&+\Delta_{n}\sum_{i=1}^{n-2}\mathrm{cos}\big(\frac{\sqrt{2u'}\Delta_{i-1}^{n}X}{\sqrt{\Delta_{n}}}\big) \mathrm{cos}\big(\frac{\sqrt{2v'}\Delta_{i+1}^{n}Y}{\sqrt{\Delta_{n}}}\big)\mathrm{cos}\big(\frac{\sqrt{2u}\Delta_{i+1}^{n}X}{\sqrt{\Delta_{n}}}\big)\mathrm{cos}\big(\frac{\sqrt{2v}\Delta_{i+2}^{n}Y}{\sqrt{\Delta_{n}}}\big)\nonumber\\
		&-3\Delta_{n}\sum_{i=1}^{n-1}\mathrm{cos}\big(\frac{\sqrt{2(u+u')}\Delta_{i}^{n}X+\sqrt{2(v+v')}\Delta_{i+1}^{n}Y}{\sqrt{\Delta_{n}}}\big)\nonumber\\
		=&\Delta_{n}\sum_{i=1}^{n-1}\mathrm{cos}\big(\frac{\sqrt{2u}\Delta_{i}^{n}X+\sqrt{2v}\Delta_{i+1}^{n}Y}{\sqrt{\Delta_{n}}}\big)\mathrm{cos}\big(\frac{\sqrt{2u'}\Delta_{i}^{n}X+\sqrt{2v'}\Delta_{i+1}^{n}Y}{\sqrt{\Delta_{n}}}\big)\nonumber\\
		&-\Delta_{n}\sum_{i=1}^{n-1}\mathrm{cos}\big(\frac{\sqrt{2(u+u')}\Delta_{i}^{n}X+\sqrt{2(v+v')}\Delta_{i+1}^{n}Y}{\sqrt{\Delta_{n}}}\big)\nonumber\\
		&+\Delta_{n}\sum_{i=1}^{n-2}\mathrm{cos}\big(\frac{\sqrt{2u}\Delta_{i-1}^{n}X}{\sqrt{\Delta_{n}}}\big) \mathrm{cos}\big(\frac{\sqrt{2v}\Delta_{i+1}^{n}Y}{\sqrt{\Delta_{n}}}\big)\mathrm{cos}\big(\frac{\sqrt{2u'}\Delta_{i+1}^{n}X}{\sqrt{\Delta_{n}}}\big)\mathrm{cos}\big(\frac{\sqrt{2v'}\Delta_{i+2}^{n}Y}{\sqrt{\Delta_{n}}}\big)\nonumber\\
		&-\Delta_{n}\sum_{i=1}^{n-1}\mathrm{cos}\big(\frac{\sqrt{2(u+u')}\Delta_{i}^{n}X+\sqrt{2(v+v')}\Delta_{i+1}^{n}Y}{\sqrt{\Delta_{n}}}\big)\nonumber\\
		&+\Delta_{n}\sum_{i=1}^{n-2}\mathrm{cos}\big(\frac{\sqrt{2u'}\Delta_{i-1}^{n}X}{\sqrt{\Delta_{n}}}\big) \mathrm{cos}\big(\frac{\sqrt{2v'}\Delta_{i+1}^{n}Y}{\sqrt{\Delta_{n}}}\big)\mathrm{cos}\big(\frac{\sqrt{2u}\Delta_{i+1}^{n}X}{\sqrt{\Delta_{n}}}\big)\mathrm{cos}\big(\frac{\sqrt{2v}\Delta_{i+2}^{n}Y}{\sqrt{\Delta_{n}}}\big)\nonumber\\
		&-\Delta_{n}\sum_{i=1}^{n-1}\mathrm{cos}\big(\frac{\sqrt{2(u+u')}\Delta_{i}^{n}X+\sqrt{2(v+v')}\Delta_{i+1}^{n}Y}{\sqrt{\Delta_{n}}}\big).\nonumber\\
	\end{align*}
Because of the Law of Large Number,
\begin{align*}
&{\tilde{\Gamma}_n}-\left[\frac{1}{2}\Delta_{n}\sum_{i=1}^{n-1}\mathrm{e}^{-u(\sigma_{i-1}^{X})^2-v(\sigma_{i-1}^{X})^{2}-u'(\sigma_{i-1}^{Y})^2-v'(\sigma_{i-1}^{Y})^2}\left[\mathrm{e}^{-2\sqrt{uu'}(\sigma_{i-1}^{X})^2+2\sqrt{uu'}(\sigma_{i-1}^{X})^2-2\sqrt{vv'}(\sigma_{i-1}^{Y})^2+2\sqrt{vv'}(\sigma_{i-1}^{Y})^2}-2\right]\right.\nonumber\\
		&+\frac{1}{2}\Delta_{n}\sum_{i=1}^{n-1}\mathrm{e}^{-u(\sigma_{i-1}^{X})^2-v(\sigma_{i-1}^{X})^{2}-u'(\sigma_{i-1}^{Y})^2-v'(\sigma_{i-1}^{Y})^2}\left[\mathrm{e}^{-2\sqrt{vu'}\rho\sigma_{i-1}^{X}\sigma_{i-1}^{Y}+2\sqrt{vu'}\rho\sigma_{i-1}^{X}\sigma_{i-1}^{Y}}-2\right]\nonumber\\
		&\left.+\frac{1}{2}\Delta_{n}\sum_{i=1}^{n-1}\mathrm{e}^{-u(\sigma_{i-1}^{X})^2-v(\sigma_{i-1}^{X})^{2}-u'(\sigma_{i-1}^{Y})^2-v'(\sigma_{i-1}^{Y})^2}\left[\mathrm{e}^{-2\sqrt{v'u}\rho\sigma_{i-1}^{X}\sigma_{i-1}^{Y}+2\sqrt{v'u}\rho\sigma_{i-1}^{X}\sigma_{i-1}^{Y}}-2\right]\right]\stackrel{P}{\longrightarrow}0,\nonumber\\
\end{align*}
and
\begin{align*}
&\frac{1}{2}\Delta_{n}\sum_{i=1}^{n-1}\mathrm{e}^{-u(\sigma_{i-1}^{X})^2-v(\sigma_{i-1}^{X})^{2}-u'(\sigma_{i-1}^{Y})^2-v'(\sigma_{i-1}^{Y})^2}\left[\mathrm{e}^{-2\sqrt{uu'}(\sigma_{i-1}^{X})^2+2\sqrt{uu'}(\sigma_{i-1}^{X})^2-2\sqrt{vv'}(\sigma_{i-1}^{Y})^2+2\sqrt{vv'}(\sigma_{i-1}^{Y})^2}-2\right]\nonumber\\
		&+\frac{1}{2}\Delta_{n}\sum_{i=1}^{n-1}\mathrm{e}^{-u(\sigma_{i-1}^{X})^2-v(\sigma_{i-1}^{X})^{2}-u'(\sigma_{i-1}^{Y})^2-v'(\sigma_{i-1}^{Y})^2}\left[\mathrm{e}^{-2\sqrt{vu'}\rho\sigma_{i-1}^{X}\sigma_{i-1}^{Y}+2\sqrt{vu'}\rho\sigma_{i-1}^{X}\sigma_{i-1}^{Y}}-2\right]\nonumber\\
		&+\frac{1}{2}\Delta_{n}\sum_{i=1}^{n-1}\mathrm{e}^{-u(\sigma_{i-1}^{X})^2-v(\sigma_{i-1}^{X})^{2}-u'(\sigma_{i-1}^{Y})^2-v'(\sigma_{i-1}^{Y})^2}\left[\mathrm{e}^{-2\sqrt{v'u}\rho\sigma_{i-1}^{X}\sigma_{i-1}^{Y}+2\sqrt{v'u}\rho\sigma_{i-1}^{X}\sigma_{i-1}^{Y}}-2\right]\nonumber\\
\stackrel{P}{\longrightarrow}&\frac{1}{2}\int_{0}^{T}\mathrm{e}^{-u(\sigma_{s}^{X})^2-v(\sigma_{s}^{X})^{2}-u'(\sigma_{s}^{Y})^2-v'(\sigma_{s}^{Y})^2}\left[\mathrm{e}^{-2\sqrt{uu'}(\sigma_{s}^{X})^2+2\sqrt{uu'}(\sigma_{s}^{X})^2-2\sqrt{vv'}(\sigma_{s}^{Y})^2+2\sqrt{vv'}(\sigma_{2}^{Y})^2}\right.\nonumber\\
		&\left.+\mathrm{e}^{-2\sqrt{vu'}\rho\sigma_{s}^{X}\sigma_{s}^{Y}+2\sqrt{vu'}\rho\sigma_{s}^{X}\sigma_{s}^{Y}}+\mathrm{e}^{-2\sqrt{v'u}\rho\sigma_{s}^{X}\sigma_{s}^{Y}+2\sqrt{v'u}\rho\sigma_{s}^{X}\sigma_{s}^{Y}}-6\right]\mathrm{d}s.\nonumber
\end{align*}
We can achieve
\begin{align*}
{\tilde{\Gamma}_n}\stackrel{P}{\longrightarrow}&\frac{1}{2}\int_{0}^{T}\mathrm{e}^{-u(\sigma_{s}^{X})^2-v(\sigma_{s}^{X})^{2}-u'(\sigma_{s}^{Y})^2-v'(\sigma_{s}^{Y})^2}\left[\mathrm{e}^{-2\sqrt{uu'}(\sigma_{s}^{X})^2+2\sqrt{uu'}(\sigma_{s}^{X})^2-2\sqrt{vv'}(\sigma_{s}^{Y})^2+2\sqrt{vv'}(\sigma_{2}^{Y})^2}\right.\nonumber\\
		&\left.+\mathrm{e}^{-2\sqrt{vu'}\rho\sigma_{s}^{X}\sigma_{s}^{Y}+2\sqrt{vu'}\rho\sigma_{s}^{X}\sigma_{s}^{Y}}+\mathrm{e}^{-2\sqrt{v'u}\rho\sigma_{s}^{X}\sigma_{s}^{Y}+2\sqrt{v'u}\rho\sigma_{s}^{X}\sigma_{s}^{Y}}-6\right]\mathrm{d}s.\nonumber
\end{align*}
\end{pot2}

\begin{pot3}
\begin{enumerate}
\item For
$$\sqrt{T}
\begin{pmatrix}
\frac{1}{T}\sum\limits_{t=1}^{T}\hat{Z}_{t}^{x,y}(u,v)-\frac{1}{T}\sum\limits_{t=1}^{T}Z_{t}^{x,y}(u,v)+\frac{1}{T}\sum\limits_{t=1}^{T}Z_{t}^{x,y}(u,v)-\mu_{t}^{x,y}(u,v)\\
\frac{1}{T}\sum\limits_{t=1}^{T}\hat{Z}_{t}^{x}(u)-\frac{1}{T}\sum\limits_{t=1}^{T}Z_{t}^{x}(u)+\frac{1}{T}\sum\limits_{t=1}^{T}Z_{t}^{x}(u)-\mu_{t}^{x}(u)\\
\frac{1}{T}\sum\limits_{t=1}^{T}\hat{Z}_{t}^{y}(v)-\frac{1}{T}\sum\limits_{t=1}^{T}Z_{t}^{y}(v)+\frac{1}{T}\sum\limits_{t=1}^{T}Z_{t}^{y}(v)-\mu_{t}^{y}(v)\\
\end{pmatrix}
$$
Combing with Assumption \ref{asu3}, Theorem \ref{thm2} and \cite{todorov2012realized}, we have
\begin{eqnarray*}
\sqrt{T}\left(\frac{1}{T}\sum\limits_{t=1}^{T}\hat{Z}_{t}^{x,y}(u,v)-\frac{1}{T}\sum\limits_{t=1}^{T}Z_{t}^{x,y}(u,v)\right)&=&O_{p}(\sqrt{T\Delta_{n}}),\\
\frac{1}{T}\sum\limits_{t=1}^{T}\hat{Z}_{t}^{x}(u)-\frac{1}{T}\sum\limits_{t=1}^{T}Z_{t}^{x}(u)&=&O_{p}(\sqrt{T\Delta_{n}})\\
\frac{1}{T}\sum\limits_{t=1}^{T}\hat{Z}_{t}^{y}(v)-\frac{1}{T}\sum\limits_{t=1}^{T}Z_{t}^{y}(v)&=&O_{p}(\sqrt{T\Delta_{n}})
\end{eqnarray*}
Then according to the Assumption \ref{asu3} and CLT for stationary processes by \cite{jacod2003limit}, the finite-dimensional convergence is as follows
$$\sqrt{T}
\begin{pmatrix}
\frac{1}{T}\sum\limits_{t=1}^{T}Z_{t}^{x,y}(u,v)-\mu_{t}^{x,y}(u,v)\\
\frac{1}{T}\sum\limits_{t=1}^{T}Z_{t}^{x}(u)-\mu_{t}^{x}(u)\\
\frac{1}{T}\sum\limits_{t=1}^{T}Z_{t}^{y}(v)-\mu_{t}^{y}(v)\\
\end{pmatrix}
\stackrel{\cal S}{\longrightarrow} \Phi(u,v),$$
where the $\Phi(u,v)$ are shown in Theorem \ref{thm3}.\\
Next, we show the tightness.
For $u_{1}, u_{2}, v_{1}, v_{2}\in\mathbb{R}_{+}$,
\begin{eqnarray*}
&&T\mathbb{E}\left[\frac{1}{T}\sum_{t=1}^{T}Z_{t}^{x,y}(u_{1},v_{1})-\mu_{t}^{x,y}(u_{1},v_{1})-\frac{1}{T}\sum_{t=1}^{T}Z_{t}^{x,y}(u_{2},v_{2})+\mu_{t}^{x,y}(u_{2},v_{2})\right]^{2}\\
&=&\frac{1}{T}\mathbb{E}\left[\sum_{t=1}^{T}\left[\left(\int_{t-1}^{t}e^{-u_{1}(\sigma_{s}^{x})^{2}-v_{1}(\sigma_{s}^{y})^{2}}-e^{-u_{2}(\sigma_{s}^{x})^{2}-v_{2}(\sigma_{s}^{y})^{2}}ds\right)\right.\right.\\
&&\left.\left.-\mathbb{E}\left(\int_{t-1}^{t}e^{-u_{1}(\sigma_{s}^{x})^{2}-v_{1}(\sigma_{s}^{y})^{2}}-e^{-u_{2}(\sigma_{s}^{x})^{2}-v_{2}(\sigma_{s}^{y})^{2}}ds\right)\right]\right]^{2}\\
&\leq&C[(u_{1}-u_{2})^{2}-(v_{1}-v_{2})^2]
\end{eqnarray*}
The last inequality follows from Assumption \ref{asu3} and Taylor formula,
\begin{eqnarray*}
&&e^{-u_{1}(\sigma_{s}^{x})^{2}-v_{1}(\sigma_{s}^{y})^{2}}-e^{-u_{2}(\sigma_{s}^{x})^{2}-v_{2}(\sigma_{s}^{y})^{2}}\\
&=&\left(-(\sigma_{s^{\star}}^{x})^{2}e^{-u_{2}(\sigma_{s^{\star}}^{x})^{2}-v_{2}(\sigma_{s^{\star}}^{y})^{2}}\right)(u_{1}-u_{2})+\left(-(\sigma_{s^{\star}}^{y})^{2}e^{-u_{2}(\sigma_{s^{\star}}^{x})^{2}-v_{2}(\sigma_{s^{\star}}^{y})^{2}}\right)(v_{1}-v_{2})
\end{eqnarray*}
where $s^{\star}\in(t-1,t)$. The desired results can be obtained from the Theorem 20 of \cite{ibragimov1981statistical}.
\item we denote
\begin{equation*}
C_{11}^{l}([u,v],[u',v'])=\frac{1}{T}\sum_{t=l+1}^{T}[Z_{t}^{x,y}(u,v)-\mu_{t}^{x,y}(u,v)][Z_{t-l}^{x,y}(u',v')-\mu_{t}^{x,y}(u',v')].
\end{equation*}
Based on the Proposition 1 in \cite{andrews1991heteroskedasticity},
\begin{equation*}
  C_{11}^{0}([u,v],[u',v'])+\sum_{i=1}^{L_{T}}\omega(i,L_{T})(C_{11}^{i}([u,v],[u',v'])+C_{11}^{i}([u',v'],[u,v]))\stackrel{P}{\longrightarrow} V_{11}([u,v],[u',v'])\\
\end{equation*}
And also $Z_{t}^{x,y}(u,v)$, $\hat{Z}_{t}^{x,y}(u,v)$ are bounded,
\begin{eqnarray}\label{equation34}
&&|\hat{C}_{11}^{i}-C_{11}^{i}|\nonumber\\
&\leq&C\frac{1}{T}\sum_{t=1}^{T}\left\{\left|\hat{Z}_{t}^{x,y}(u,v)-Z_{t}^{x,y}(u,v)\right|+\left|\frac{1}{T}\sum_{t=1}^{T}\hat{Z}_{t}^{x,y}(u,v)-\mu_{t}^{x,y}(u,v)\right|\right\}\nonumber\\
&\leq&C\frac{1}{T}\sum_{t=1}^{T}\left\{\left|\hat{Z}_{t}^{x,y}(u,v)-Z_{t}^{x,y}(u,v)\right|+\left|\frac{1}{T}\sum_{t=1}^{T}\hat{Z}_{t}^{x,y}(u,v)-\frac{1}{T}\sum_{t=1}^{T}Z_{t}^{x,y}(u,v)\right|\right.\nonumber\\
&&\left.+\left|\frac{1}{T}\sum_{t=1}^{T}Z_{t}^{x,y}(u,v)-\mu_{t}^{x,y}(u,v)\right|\right\}
\end{eqnarray}
The first and second terms of (\ref{equation34}) are $O_{p}(\sqrt{\Delta_{n}})$ according to  Theorem \ref{thm2}, the magnitude of the last term is $O_{p}(\frac{1}{\sqrt{T}})$ due to Theorem \ref{thm3}. Therefore, the relationship between $T$,$\Delta_{n}$ and $L_{T}$ can be obtained from
$\sqrt{\Delta_{n}L_{T}}\to 0,  L_{T}\frac{1}{\sqrt{T}}\to 0$. Then, we use the same way to achieve the consistency of $\hat{V}_{ab}^{i}$, $a=1,2,3$, $b=1,2,3$.
\end{enumerate}
\end{pot3}



\end{document}